\newif\ifisarxiv
\isarxivtrue

\documentclass[10pt]{article} 
\ifisarxiv
\usepackage{fullpage}
\else
\usepackage[accepted]{tmlr}
\fi

\usepackage{hyperref}       
\usepackage{url}            
\usepackage{booktabs}       
\usepackage{nicefrac}       
\usepackage{xcolor}         

\usepackage{amsmath, amssymb,  graphicx, url, algorithm2e}
\usepackage{subfigure}
\usepackage{tikz,pgfplots}
\usepackage{enumitem}
\usepackage{algorithmicx,algpseudocode}
\pgfplotsset{compat=newest}
\usepackage{xcolor}
\hypersetup{
	colorlinks,
	linkcolor={red!40!gray},
	citecolor={blue!40!gray},
	urlcolor={blue!70!gray}
}

\def\gbh {\widehat{\g}}
\def\gbt {\tilde{\g}}
\def\pbt {\tilde{\p}}

\def\cmark{\Green{\checkmark}}
\def\xmark{\Red{\large\sffamily x}}

\newif\ifDRAFT
\DRAFTtrue
\ifDRAFT
\newcommand{\marrow}{\marginpar[\hfill$\longrightarrow$]{$\longleftarrow$}}
\newcommand{\niceremark}[3]
   {\textcolor{red}{\textsc{#1 #2:} \marrow\textsf{#3}}}
\newcommand{\ken}[2][says]{\niceremark{Ken}{#1}{#2}}

\newcommand{\michael}[2][says]{\niceremark{Michael}{#1}{#2}}
\newcommand{\michal}[2][says]{\niceremark{Michal}{#1}{#2}}
\newcommand{\feynman}[2][says]{\niceremark{Feynman}{#1}{#2}}
\else
\newcommand{\ken}[1]{}
\newcommand{\michael}[1]{}
\newcommand{\michal}[1]{}
\newcommand{\feynman}[1]{}
\fi

\def\ee{\mathrm{e}}

\def\checkmark{\tikz\fill[scale=0.4](0,.35) -- (.25,0) --
(1,.7) -- (.25,.15) -- cycle;}

\def\xib{\boldsymbol\xi}

\def\S{\mathbf{S}}

\def\xbt{\tilde{\x}}

\def\g{{\mathbf{g}}}

\def\Nc{\mathcal{N}}

\def\p{\mathbf p}

\def\H{\mathbf H}
\def\Hbh{\widehat{\H}}
\def\Hbt{\tilde{\H}}

\ifx\BlackBox\undefined
\newcommand{\BlackBox}{\rule{1.5ex}{1.5ex}}  
\fi
\DeclareMathOperator*{\argmin}{\mathop{\mathrm{argmin}}}

\def\x{\mathbf x}
\def\y{\mathbf y}

\def\z{\mathbf z}
\def\a{\mathbf a}

\def\v{\mathbf v}

\def\zero{\mathbf 0}

\def\B{\mathbf B}
\def\A{\mathbf A}
\def\C{\mathbf C}
\def\U{\mathbf U}

\def\G{\mathbf G}
\def\D{\mathbf D}
\def\V{\mathbf V}
\def\M{\mathbf M}

\def\Dc{\mathcal{D}}
\def\Z{\mathbf Z}

\def\I{\mathbf I}

\def\A{\mathbf A}

\def\E{\mathbb E}
\def\R{\mathbb R}

\def\Pr{\mathrm{Pr}}

\let\origtop\top
\renewcommand\top{{\scriptscriptstyle{\origtop}}} 

\definecolor{silver}{cmyk}{0,0,0,0.3}
\definecolor{yellow}{cmyk}{0,0,0.9,0.0}
\definecolor{reddishyellow}{cmyk}{0,0.22,1.0,0.0}
\definecolor{black}{cmyk}{0,0,0.0,1.0}
\definecolor{darkYellow}{cmyk}{0.2,0.4,1.0,0}
\definecolor{orange}{cmyk}{0.0,0.7,0.9,0}
\definecolor{darkSilver}{cmyk}{0,0,0,0.1}
\definecolor{grey}{cmyk}{0,0,0,0.5}
\definecolor{darkgreen}{cmyk}{0.9,0,1,0} 
\newcommand{\Red}[1]{{\color{red}  {#1}}}

\newcommand{\Green}[1]{{\color{darkgreen}  {#1}}}

\newcommand{\white}[1]{{\textcolor{white}{#1}}}

\ifx\proof\undefined
\newenvironment{proof}{\par\noindent{\bf Proof\ }}{\hfill\BlackBox\\[2mm]}
\fi

\ifx\theorem\undefined
\newtheorem{theorem}{Theorem}
\fi

\ifx\example\undefined
\newtheorem{example}{Example}
\fi

\ifx\condition\undefined
\newtheorem{condition}{Condition}
\fi
\ifx\property\undefined
\newtheorem{property}{Property}
\fi

\ifx\lemma\undefined
\newtheorem{lemma}{Lemma}
\fi

\ifx\proposition\undefined
\newtheorem{proposition}{Proposition}
\fi

\ifx\remark\undefined
\newtheorem{remark}{Remark}
\fi

\ifx\corollary\undefined
\newtheorem{corollary}{Corollary}
\fi

\ifx\definition\undefined
\newtheorem{definition}{Definition}
\fi

\ifx\conjecture\undefined
\newtheorem{conjecture}{Conjecture}
\fi

\ifx\axiom\undefined
\newtheorem{axiom}{Axiom}
\fi

\ifx\claim\undefined
\newtheorem{claim}{Claim}
\fi

\ifx\assumption\undefined
\newtheorem{assumption}{Assumption}
\fi

\ifx\condition\undefined
\newtheorem{condition}{Condition}
\fi

\def\ch{{\alpha}}

\title{Stochastic Variance-Reduced Newton:\\ Accelerating
  Finite-Sum Minimization with Large Batches}

\author{%
  \textbf{Micha{\l } Derezi\'{n}ski}\\
  Department of Electrical Engineering \& Computer Science\\
  University of Michigan\\
  \texttt{derezin@umich.edu}\\
}
\date{}

\begin{document}

\maketitle

\begin{abstract}
  Stochastic variance reduction has proven effective at accelerating
  first-order algorithms for solving convex finite-sum optimization
  tasks such as empirical risk minimization. Incorporating
  second-order information has proven helpful in further improving the
  performance of these first-order methods. Yet, comparatively little
  is known about the benefits of using variance reduction to
  accelerate popular stochastic second-order methods such as
  Subsampled Newton. To address this, we propose Stochastic
  Variance-Reduced Newton (SVRN), a finite-sum minimization algorithm
  that provably accelerates existing stochastic Newton methods from
  $O(\alpha\log(1/\epsilon))$ to
  $O\big(\frac{\log(1/\epsilon)}{\log(n)}\big)$ passes over the data,
  i.e., by a factor of $O(\alpha\log(n))$, where $n$ is the number of
  sum components and $\alpha$ is the approximation factor in the
  Hessian estimate. Surprisingly, this acceleration gets more significant the larger the
data size $n$, which is a unique property of SVRN. Our algorithm retains the key advantages of Newton-type methods, such as easily parallelizable large-batch operations and a simple unit step size. We use SVRN to accelerate Subsampled Newton and Iterative Hessian Sketch algorithms, and show that it compares favorably to popular first-order methods with variance~reduction.
\end{abstract}

\section{Introduction}
\label{s:intro}

Consider a convex finite-sum minimization task:
\begin{align}
  \text{find}\quad \x^*=\argmin_{\x\in\R^d} f(\x)\quad \text{ for }\quad  f(\x) =
  \frac1n\sum_{i=1}^n\psi_i(\x).\label{eq:finite-sum}
\end{align}
This optimization task naturally arises in machine learning through empirical risk minimization,
where $\x$ is the model parameter vector and each function $\psi_i(\x)$
corresponds to the loss incurred by the model on the $i$-th element in a
training data set (e.g., square loss for regression, or logistic loss
for classification). Many other optimization tasks, such as solving
semi-definite programs and portfolio optimization, can be cast in this
general form. Our goal is to find an $\epsilon$-approximate solution,
i.e., $\xbt$ such that $f(\xbt) - f(\x^*) \leq \epsilon$.

Classical iterative optimization methods (such as gradient descent and
Newton's method) use first/second-order information of function $f$ to construct a sequence
$\x_0,\x_1,...$ that converges to $\x^*$. However, this does not leverage the finite-sum structure of the
problem. Thus, extensive literature has been dedicated to efficiently solving
finite-sum minimization tasks using stochastic optimization methods,
which use first/second-order information of randomly sampled
component functions $\psi_i$, that can often be computed much faster
than the entire function $f$. Among first-order methods,
variance-reduction techniques such as SAG \cite{roux2012stochastic},
SDCA \cite{shalev2013stochastic}, SVRG \cite{johnson2013accelerating}, SAGA
\cite{defazio2014saga}, Katyusha \cite{allen2017katyusha} and others \cite{frostig2015competing, konecny2015mini,allen2016improved}, have proven particularly
effective. One of the most popular variants of this approach is
Stochastic Variance-Reduced Gradient (SVRG),
which achieves variance reduction by combining frequent stochastic gradient queries with
occasional full batch gradient queries, to optimize the overall cost
of finding an $\epsilon$-approximate solution, where the cost is
measured by the total number of queries to the components $\nabla\psi_i(\x)$.

Many stochastic second-order methods have also been proposed for solving finite-sum
minimization, including Subsampled Newton
\cite{erdogdu2015convergence,roosta2019sub,bollapragada2018exact,berahas2020investigation,determinantal-averaging}
Newton Sketch
\cite{pilanci2016iterative,pilanci2017newton,newton-less},
and others
\cite{kovalev2019stochastic,moritz2016linearly,tripuraneni2018stochastic,mokhtari2018iqn,gupta2021localnewton}. 
 These approaches are generally less
sensitive to hyperparameters such as the step size, and they typically
query larger random batches of component gradients/Hessians at a time,
as compared to first-order methods. The larger queries make
these methods less sequential, allowing for more effective
vectorization and~parallelization.

A number of works 
have explored whether second-order information can be used to
accelerate stochastic variance-reduced methods, resulting in several
algorithms such as Preconditioned SVRG \cite{gonen2016solving}, SVRG2
\cite{gower2018tracking} and others \cite{gower2016stochastic,liu2019acceleration}. However, these are still primarily stochastic first-order
methods, highly sequential and with a problem-dependent step
size. Comparatively little work has been done on using variance
reduction to accelerate stochastic Newton-type methods for convex
finite-sum minimization (see discussion in Section~\ref{s:related-work}).
To that end, we ask:

\begin{center}
\emph{Can variance reduction accelerate local convergence
 of Stochastic Newton\\ in convex finite-sum minimization?}
\end{center}

We show that the answer to this question is positive. The method that
we use to demonstrate this, which we call Stochastic Variance-Reduced Newton
(SVRN), retains the positive characteristics of
second-order methods, including easily parallelizable large-batch
gradient queries, as well as minimal
hyperparameter tuning (e.g., accepting a unit step size). We prove that, when the number of components
$\psi_i$ is sufficiently large, SVRN achieves a better parallel 
complexity than SVRG, and a better sequential complexity than the corresponding
Stochastic Newton method (see Table \ref{tab:svrn}).

\section{Main result}\label{s:main}
In this section, we present our main result, which is the parallel
and sequential 
complexity analysis of the local convergence of SVRN. The algorithm itself is
discussed in detail in Section \ref{s:proof}.

We now present the assumptions needed for our main result,
starting with $\mu$-strong convexity of $f$ and $\lambda$-smoothness
of each $\psi_i$.  These are standard for establishing linear convergence rate of SVRG.
  Our result also requires Hessian regularity assumptions
  (Definition~\ref{d:neighborhood}), which are standard for 
Newton's method and only affect the size of the local convergence 
neighborhood.
\begin{assumption}\label{a:convex}
We assume that $f(x)=\frac1n\sum_{i=1}^n\psi_i(x)$ has continuous first and second derivatives, as well as a bounded condition number $\kappa=\lambda/\mu$, where $\mu$ and $\lambda$ are
  defined as follows:
\vspace{-3mm}
  \begin{enumerate}
  \item Function $f$ is $\mu$-strongly convex, i.e.,\vspace{-2mm}
    \begin{align*}
      f(\x)\geq f(\x')+\nabla
      f(\x')^\top(\x-\x') + \frac\mu2\|\x-\x'\|^2;
    \end{align*}
    \vspace{-8mm}
    \item Each of the $n$ components $\psi_i$ is $\lambda$-smooth,
      i.e.,\vspace{-2mm}
      \begin{align*}
        \psi_i(\x)\leq \psi_i(\x') +
        \nabla\psi_i(\x')^\top(\x-\x')+\frac \lambda2\|\x-\x'\|^2.
      \end{align*}
   \end{enumerate}
 \end{assumption}

 To highlight the parallelizability of SVRN due to large mini-batches, as well as the effect of
 variance reduction on its performance, we will consider two standard
 complexity measures:
 \begin{enumerate}
 \item \textbf{Parallel complexity:} Number of batch gradient queries,
   i.e., times the algorithm
   computes a gradient at an iterate, over the full batch or a
   mini-batch. This corresponds to the PRAM model.
 \item \textbf{Sequential complexity:} Number of queries to component
   gradients $\nabla\psi_i(\x)$, normalized by the total number of
   components $n$. Thus, one can think of this as the number of ``data
   passes'', where one ``data pass'' may possibly query component gradients at
   different iterates. This is a natural measure of complexity for
   stochastic first-order methods, which coincides with the parallel
   complexity when we only query full-batch gradients.
 \end{enumerate}

The motivation for this dual-complexity perspective is centered on the benefits of using large-batch gradient queries.
These benefits come from the fact that on most computing architectures it takes far
 less time to compute a single gradient estimate on a batch of
 component functions at one location, than it takes to compute
 component gradients at  different locations, one at a time. We
 highlight this by the notion of parallel complexity, which measures the
 number of batch gradient queries. Note that the benefits of large
 mini-batch sizes are not limited to parallel computing. Even standard
single- and  multi-core architectures benefit substantially from the
 vectorization of gradient computations, which is only effective when
 using large batches (as we observe in our experiments, see Section
 \ref{s:experiments}).

 Despite these benefits,  parallelization and
 vectorization still come with some computation/communication
 overhead, which is why it is natural to also optimize over the sequential
 complexity of the algorithms. Altogether, to express this in our
model, we optimize both over parallel and sequential complexity, while
prioritizing the parallel one. Note that for any algorithm using
only full gradients (such as the standard versions of gradient descent
or Subsampled Newton), the two notions of complexity are exactly
equivalent. For example, in gradient descent (GD), both parallel and
sequential complexity is $O(\kappa\log(1/\epsilon))$. By introducing stochastic gradients and variance reduction,
as in SVRG, we can improve upon the sequential complexity of GD, while
preserving (but not improving)
its parallel complexity. Specifically, when
$n\gg \kappa$, then SVRG with optimal mini-batch size takes
$O(\log(1/\epsilon))$ sequential time to find
an $\epsilon$-approximate solution, however it still needs
$O(\kappa\log(1/\epsilon))$ parallel time (Table~\ref{tab:svrn}).

We can avoid the dependence of parallel complexity on the condition number
$\kappa$ by using second-order information. In particular, suppose
that in each iteration we compute the full gradient $\nabla f(\x)$ and
are given access to a (typically stochastic) Hessian estimate $\Hbt$ such that for some
$1\leq \alpha\ll\kappa$,
\begin{align}
\text{(Hessian $\alpha$-approximation)}\qquad  \frac1{\sqrt\alpha}\, \nabla^2f(\x)\preceq
\Hbt\preceq\sqrt\alpha\,\nabla^2f(\x),\label{eq:alpha}
\end{align}
In fact, our condition can be stated in an even more general way, by asking that $\beta_1\nabla^2 f(x)\preceq \Hbt
\preceq \beta_2\nabla^2 f(x)$ for some fixed $0<\beta_1\leq
\beta_2$ such that $\alpha = \beta_2/\beta_1$. In this case, $(1/\sqrt{\beta_1\beta_2})\Hbt$ is an
$\alpha$-approximation in the sense of \eqref{eq:alpha}, which means that we can apply our results and
algorithms given any such Hessian estimates.

A standard Stochastic Newton (SN) update, given below in
\eqref{eq:stochastic-newton}, can achieve parallel and sequential 
complexity of $O(\alpha\log(1/\epsilon))$ locally in the neighborhood
of the optimum. It is thus natural to ask
whether we can use stochastic gradients and variance reduction to
accelerate the local sequential complexity of this method, while
preserving its parallel complexity. Our main result shows not only that this is
possible, but also, remarkably, that this acceleration gets more significant the larger
the data size $n$. See also Theorem \ref{t:svrn} for algorithmic
details and convergence analysis.

\begin{theorem}[informal Theorem \ref{t:svrn}]\label{t:informal}
Suppose that Assumption \ref{a:convex} holds and: (a)
$f$ has a Lipschitz Hessian, or (b) $f$ is
self-concordant. Moreover, let $n\gg\kappa\gg\alpha$. There is an algorithm (SVRN)
and an open neighborhood $U$ such
that, given any $\x\in U$ with a corresponding Hessian $\alpha$-approximation as
in \eqref{eq:alpha}, the cost of returning
$\xbt$ such that $f(\xbt)-f(\x^*)\leq \epsilon\cdot
(f(\x)-f(\x^*))$ is as follows:
\begin{align*}
\textnormal{Parallel time}\  =\  O\big(\alpha\log(1/\epsilon)\big)\text{
  batch queries}\qquad
  \text{and}\qquad
  \textnormal{Sequential time}
\ =\   O\Big(\frac{\log(1/\epsilon)}{\log(n)}\Big)\text{ data passes.}
\end{align*}
\end{theorem}
\begin{table*}[t]
\centering  \begin{tabular}{r||c|c|c|c}
  & Second-order 
  & Parallel (batch queries)
  & Sequential (data passes)
  \\
  \hline\hline
Gradient Descent
  & \xmark 
  &$O(\Red{\kappa}\log(1/\epsilon))$
  &$O(\Red{\kappa}\log(1/\epsilon))$
  \\[1mm]
Accelerated GD
  & \xmark 
  &$O(\Red{\sqrt \kappa}\log(1/\epsilon))$
  &$O(\Red{\sqrt \kappa}\log(1/\epsilon))$
  \\[1mm]
Stochastic Newton 
  & \cmark 
  & $O(\Red{\alpha}\log(1/\epsilon))$
  & $O(\Red{\alpha}\log(1/\epsilon))$
  \\[1mm]
 Mini-batch SVRG 
  & \xmark 
  &$O(\Red{\kappa}\log(1/\epsilon))$
  &$O(\log(1/\epsilon))$
  \\[1mm]
Mini-batch Katyusha 
  &  \xmark 
  &$O(\Red{\sqrt\kappa}\log(1/\epsilon))$
  &$O(\log(1/\epsilon))$
  \\[1mm]
  \hline
  \textbf{SVRN (this work)}
  & \cmark 
  & $O(\Red{\alpha}\log(1/\epsilon))$
  & $O\big(\frac{\log(1/\epsilon)}{\Green{\log(n)}}\big)\white{\Big|}$
\end{tabular}
\caption{Comparison of local convergence behavior for SVRN and
  related stochastic methods in the big data regime, i.e., $n\gg
  \kappa$, along with full-batch Gradient Descent (GD), and
  Accelerated GD. 
Time complexities are obtained by first optimizing parallel
time (batch queries), and then optimizing sequential time (data passes). For the second-order
methods, we assume a Hessian $\alpha$-approximation \eqref{eq:alpha}
where $\alpha\ll\kappa$.
} 
\label{tab:svrn}
\end{table*}
\begin{remark}
 SVRN improves on the
 sequential complexity of Stochastic Newton by $O(\alpha\log(n))$, while retaining the
 same parallel complexity. Moreover, if $\alpha\leq 2$,
 then the algorithm accepts a unit step size, and still
 achieves $O(\log(n))$ acceleration.
 Note that this acceleration improves with the problem size $n$, which
 is a unique property of SVRN.  
\end{remark}
\begin{remark}
  To find an initialization point for SVRN, one can simply run
  a few iterations of a Subsampled Newton method with line
search. In Section \ref{s:svrn-ha}, we propose a globally
convergent algorithm based on this approach (SVRN-HA; see
Algorithm~\ref{alg:svrn}), and in Section \ref{s:experiments} we show
empirically that it substantially accelerates Subsampled Newton.
\end{remark}

\subsection{Discussion}
\label{s:discussion}
In this section, we compare the local convergence complexity of SVRN
to standard stochastic first-order 
and second-order algorithms. In this comparison, we focus on what we
call the big data regime, i.e., $n\gg \kappa$, which is of primary
interest in the literature on Subsampled Newton methods.
Then, in Section~\ref{s:ls},
we illustrate how SVRN can be further improved via sketching and importance
sampling, when solving problems with additional structure, such as least
squares. 

\paragraph{Comparison to SVRG and Katyusha.}
As we can see in Table \ref{tab:svrn}, first-order algorithms,
including variance-reduced methods such as SVRG \cite{johnson2013accelerating}, and its accelerated
variants like Katyusha \cite{allen2017katyusha}, suffer
from a dependence on the condition number $\kappa$ in their parallel
complexity. Namely, they require either $O(\kappa\log(1/\epsilon))$ or $O(\sqrt
\kappa\log(1/\epsilon))$ batch gradient queries, compared to 
$O(\alpha\log(1/\epsilon))$ for SVRN and Stochastic Newton, where
$\alpha$ is the
Hessian approximation factor, which is often much smaller than $\sqrt\kappa$. This is because,
unlike SVRN, these methods do not scale well to large mini-batches,
making them less parallelizable. 

Another difference between SVRN and SVRG or Katyusha is that, when
the Hessian approximation is sufficiently accurate ($\alpha\leq 2$), then
SVRN accepts a unit
step size, which leads to optimal
convergence rate without any tuning. On the other hand, the
optimal step size for SVRG depends on the strong convexity and
smoothness constants $\mu$ and $\lambda$, and thus, requires tuning.

\paragraph{Comparison to Stochastic Newton.}
We next compare SVRN with Stochastic Newton
methods such as Subsampled Newton and Newton Sketch. Here, the most standard proto-algorithm 
considered in the literature is the following update:
 \begin{align}
   \xbt_{s+1} = \xbt_s - \eta_s\Hbt^{-1}\,\nabla f(\xbt_s).\label{eq:stochastic-newton}
 \end{align}
As mentioned earlier, this update uses only full gradients, so both
its parallel and sequential complexity is $O(\alpha\log(1/\epsilon))$ (see Stochastic Newton in
Table \ref{tab:svrn}). On the other hand, SVRN provides a direct
acceleration of the sequential complexity without sacrificing any
parallel complexity.

The Hessian $\alpha$-approximation $\Hbt\approx \nabla^2 f(\xbt_s)$ can be
produced in a number of ways, but perhaps the most relevant for this discussion is Hessian
subsampling, a.k.a. Subsampled Newton \cite{roosta2019sub}. 
In this setting, given $k$ component Hessians sampled uniformly at
random, we can with high probability construct an estimate $\Hbt$ with
approximation factor $\alpha = 1+O\big(\kappa\log(d)/k +
\sqrt{\kappa\log(d)/k}\big)$ (see
Appendix~\ref{a:matrix-bernstein}), so that $\alpha\ll\kappa$ for any
$k\gg \log(d)$, and $\alpha\leq 2$ for $k = O(\kappa\log (d))$.
However, if we wanted to recover
SVRN's sequential complexity of $O\big(\frac{\log(1/\epsilon)}{\log(n)}\big)$
purely by improving the Hessian approximation in Subsampled Newton, the required
Hessian sample size $k$ would become at least as large as $n$, meaning
that we would essentially have to use the exact Hessian (i.e.,
Newton's method), which is
highly undesirable. 

We note that some of the literature on
Subsampled Newton proposes to subsample both the Hessian and the
gradient \cite{bollapragada2018exact}, which would be akin to $\xbt_{s+1} = \xbt_s -
\eta_s\Hbt^{-1}\frac1m\sum_{i=1}^m\nabla \psi_{I_i}(\xbt_s)$. However, as is
noted in the literature, to maintain linear convergence of such a
method, one has to keep increasing the gradient sample size at an exponential
rate, which means that, for finite-sum minimization, we quickly revert
back to the full~gradient (see experiments in Section \ref{s:experiments-ls}).  

\subsection{Accelerating SVRN with sketching and importance sampling}
\label{s:ls}
When the minimization task possesses additional structure, then we can
combine SVRN with Hessian and gradient estimation techniques other than
uniform subsampling. For example, one such family of techniques, called randomized
sketching \cite{DM16_CACM,woodruff2014sketching,randlapack_book_v1,derezinski2024recent}, is applicable when the Hessian can be represented by a
decomposition $\nabla^2f(\x)=\A_f(\x)^\top\A_f(\x)+\C$, where $\A_f(\x)$
is a tall $n\times d$ matrix and $\C$ is a fixed $d\times d$
matrix. This setting applies for many empirical risk minimization
tasks, including linear and logistic regression, among
others.

Sketching can be used to construct an estimate of 
the Hessian by applying a randomized linear transformation to
$\A_f(\x)$, represented by a $k\times n$ random matrix $\S$, where
$k = \tilde O(d)$ is much smaller than $n$. Using standard
sketching techniques, such as Subsampled Randomized Hadamard
Transforms \cite[SRHT]{ailon2009fast}, Sparse Johnson-Lindenstrauss
Transforms \cite[SJLT]{cw-sparse,nn-sparse,chenakkod2023optimal} and Leverage Score
Sparsified embeddings \cite[LESS]{less-embeddings}, we
can construct a Hessian estimate  
that satisfies the requirements of Theorem~\ref{t:informal} at the cost of
$\tilde O(nd + d^3)$, which corresponds to
a nearly-constant number of data passes and $d\times d$ matrix multiplies. In particular, this eliminates the dependence of
Hessian estimation on the condition number.

Another way of
making SVRN more efficient is to use importance sampling in the
stochastic gradient mini-batches. Importance sampling can be introduced to any finite-sum
minimization task \eqref{eq:finite-sum} by specifying an $n$-dimensional probability vector
$p=(p_1,...,p_n)$, such that $\sum_ip_i=1$, and sampling the component
gradient $\psi_{I_i}(\x)$ so that the index
$I_i$ is drawn according to $p$.
With the right choice of
importance sampling, we can substantially reduce the smoothness
parameter $\lambda$, and thereby, the condition number $\kappa$ of
the problem (see Appendix \ref{a:ls}).

The above techniques can be used to
accelerate SVRN, for instance, in the important task of solving least
squares regression. Here, given an $n\times d$ matrix $\A$ with rows $\a_i^\top$ and an
  $n$-dimensional vector $\y$, the objective being minimized is the
  following quadratic function:
\begin{align}
  f(\x) = \frac1{2n}\|\A\x-\y\|^2 = \frac1{n}\sum_{i=1}^n\frac12(\a_i^\top\x-y_i)^2.\label{eq:ls-first}
\end{align}

  One of the popular
  methods for solving the least squares task, known as the Iterative
  Hessian Sketch \cite[IHS]{pilanci2016iterative}, is exactly the Stochastic Newton update
  \eqref{eq:stochastic-newton}, where the Hessian estimate
  $\Hbt$ is constructed via sketching. In this context, SVRN can be
  viewed as an accelerated version of IHS. To
  fully leverage the structure of the least squares problem,
  we use a popular
  importance sampling technique called leverage score sampling \cite{drineas2006sampling,fast-leverage-scores}, where
  the importance probabilities are (approximately) proportional to
  $p_i\propto \a_i^\top(\A^\top\A)^{-1}\a_i$. Through an adaptation of
  our main result, we show that a version of SVRN for least squares,
  with sketched Hessian and leverage score 
  sampled gradients, improves on the state-of-the-art complexity for reaching
  a high-precision solution to a preconditioned least squares task   from
  $O(nd\log(1/\epsilon))$ \cite{rokhlin2008fast,avron2010blendenpik,meng2014lsrn}
  to $O\big(nd\,\frac{\log(1/\epsilon)}{\log(n/d)}\big)$. See
  Appendix~\ref{a:ls} for proof and further discussion. 
  \begin{theorem}[Fast least squares solver]
    \label{c:least-squares}
   Given $\A\in\R^{n\times d}$ and $\y\in\R^n$,
   after $O(nd\log n + d^3\log d)$ preprocessing cost to find the sketched
   Hessian estimate $\Hbt$ and an approximate
   leverage score distribution, SVRN finds $\xbt$ so
   that
   \begin{align*}
     f(\xbt)\leq(1+\epsilon)
   f(\x^*)\quad\text{ in }\quad
     O\Big(nd\,\tfrac{\log(1/\epsilon)}{\log(n/d)}\Big)\quad\text{ time.}
     \end{align*}
 \end{theorem}
 Crucially, the
SVRN-based least squares solver only requires a 
 preconditioner $\Hbt$ that is a 
constant factor approximation of the Hessian, i.e., $\alpha=O(1)$. Interestingly,
our approach of transforming the problem via leverage score sampling
appears to be connected to a weighted and preconditioned SGD algorithm of \cite{yang2017weighted} for
solving a more general class of $\ell_p$-regression
problems. We expect that Theorem \ref{c:least-squares}
can be similarly extended beyond least squares regression.

\subsection{Further related work}
\label{s:related-work}

As mentioned in Section \ref{s:intro}, a number of works have aimed to
accelerate first-order variance reduction methods by preconditioning
them with second-order information. For example, \cite{gonen2016solving} proposed
Preconditioned SVRG for ridge regression. The effect of this preconditioning, as in
related works \cite{liu2019acceleration}, is a reduced condition number
$\kappa$ of the problem. This is different from Theorem~\ref{t:informal},
which uses preconditioning to make variance reduction effective for
large mini-batches and with a unit step~size.

Some works have shown that, instead of preconditioning, one can use
momentum to accelerate variance reduction, and also to improve its
convergence rate when using mini-batches. These methods include
Catalyst \cite{lin2015universal} and Katyusha
\cite{allen2017katyusha}. However, unlike SVRN, these approaches are
still limited to fairly small mini-batches, as demonstrated in Table \ref{tab:svrn}.

A number of works have proposed applying techniques inspired by
variance reduction to stochastic Newton-type methods in settings
which are largely incomparable to ours. First,
\cite{rodomanov2016superlinearly,kovalev2019stochastic} consider
algorithms where the Hessian and gradient information is incrementally updated 
with either individual samples or mini-batches. However, the
approximate Hessian information required by these methods is quite
different than the one used in SVRN: they
require the Hessian estimate to be initialized with \emph{all} $n$
component Hessians, possibly computed at different locations (compared
to, e.g., a subsampled estimate). For
example, in the case of least squares, this means computing
the exact Hessian, which costs $O(nd^2)$ time and renders the task
trivial (compare this to our Theorem~\ref{c:least-squares}, where the Hessian estimate
required by SVRN can be approximated efficiently). Setting this aside,
we can still compare the local convergence rate of SVRN with the Stochastic
Newton method of \cite[Theorem 1]{kovalev2019stochastic} using the same
mini-batch size. Assuming $n\gg \kappa$ and using the setup from
Theorem~\ref{t:informal},
their method achieves
$O(\log(1/\epsilon))$
data passes, whereas SVRN
obtains the accelerated complexity of $O\big(\frac{\log(1/\epsilon)}{\log(n)}\big)$
data passes.

In the non-convex setting, variance reduction
was used by \cite{zhou2019stochastic,zhang2022adaptive} to accelerate
Subsampled Newton with cubic 
regularization. They use variance reduction both for
the gradient and the Hessian estimates. Also, \cite{wang2017stochastic} incorporate variance reduction into a stochastic quasi-Newton method. However, due to the non-convex setting, these results are
incomparable to ours, as we are focusing on strongly convex~optimization.

\section{Local convergence analysis of SVRN}
\label{s:proof}
In this section, we present the convergence analysis for SVRN, leading
to the proof of Theorem \ref{t:informal}.
\paragraph{Notation.}
For $d\times d$ positive semidefinite matrices $\A$ and $\B$, we define
$\|\v\|_{\A}=\sqrt{\v^\top\A\v}$, and we say that $\A\approx_\epsilon\B$, when
$(1-\epsilon)\B\preceq\A\preceq(1+\epsilon)\B$, where $\preceq$
denotes the positive semidefinite ordering (we define analogous
notation $a\approx_\epsilon b$ for non-negative scalars
$a,b$). We use $c$ and $C$ to denote positive absolute constants, and let $I\sim[n]$ denote a uniformly random sample from $\{1,..,n\}$.

We will present the analysis
in a slightly more general setting of expected risk minimization,
i.e., where $f(\x) = \E_{\psi\sim \Dc}[\psi(\x)]$.
Here, $\Dc$ is a distribution over convex functions
$\psi:\R^d\rightarrow \R$. Clearly, this setting subsumes
\eqref{eq:finite-sum}, since we can let $\Dc$ be a uniformly random
sample $\psi_i$. Thanks to this extension, our results can apply
to \textit{importance} sampling of component
functions, as in Theorem~\ref{c:least-squares}.
\begin{definition}\label{d:neighborhood}
We define the local convergence neighborhood $U_f(\epsilon)$, 
parameterized by $\epsilon\in(0,1)$, as:
\vspace{-3mm}
\begin{enumerate}
\item If $f$ has an $L$-Lipschitz Hessian,
    then $U_f(\epsilon)= \{\x :
    \|\x-\x^*\|_{\nabla^2f(\x^*)}<\epsilon\, \mu^{3/2}/L\}$;
  \item If $f$ is self-concordant,
    then we use
    $U_f(\epsilon)=\{\x:\|\x-\x^*\|_{\nabla^2f(\x^*)}<\epsilon/4\}$.
  \end{enumerate}
\end{definition}
Our local convergence analysis is captured by the following
theorem, which provides the rate of convergence after one outer
iteration of SVRN (stated below), for a range of
mini-batch sizes $m$.
\begin{theorem}[Convergence rate of SVRN]\label{t:svrn}
Suppose that Assumption \ref{a:convex} holds, $\alpha\geq 1$, and either: (a)
$f$ has a Lipschitz continuous Hessian, or (b) $f$ is
self-concordant.
There is an absolute constant $c>0$ such that
if $\xbt_s\in U_f(1/c\alpha)$, and we
are given the gradient $\gbt_s=\nabla f(\xbt_s)$
as well as a Hessian $\ch$-approximation, i.e., $\Hbt$ such that
$\frac1{\sqrt \ch}\nabla^2f(\xbt_s)\preceq\Hbt\preceq \sqrt\ch\,\nabla^2f(\xbt_s)$,
then, letting $\x_0=\xbt_s$ and: 
\begin{align*}
  \x_{t+1} = \x_t - \eta\Hbt^{-1}\bigg(\frac1m\sum_{i=1}^m\nabla\psi_i(\x_t) -
  \nabla\psi_i(\xbt_s) + \gbt_s\bigg),\qquad \psi_1,...,\psi_m\sim\Dc,
\end{align*}
after $t$ iterations with mini-batch size $m\geq
c\ch^2\kappa \log(t/\delta)$ and step size $\eta=\min\{\sqrt{2/\alpha},1\}$, 
the iterate $\xbt_{s+1}=\x_t$ (i.e., one outer iteration of SVRN) with probability
$1-\delta$ satisfies:
\begin{align*}
  \frac{f(\xbt_{s+1})-f(\x^*)}{f(\xbt_s)-f(\x^*)} \leq 
\Big(1-\frac1{2\alpha}\Big)^{t} + c \ch^2\log(t/\delta)\frac{\kappa } m.
\end{align*}
\end{theorem}
\begin{remark}
The dependence on the condition number $\kappa$ in our result comes
from sub-sampled gradient estimation. This is consistent with
Subsampled Newton works such as \cite{roosta2019sub}: to recover their fast condition number-free convergence
guarantees they require the gradient sample size to be sufficiently
larger than $\kappa$.  We showed how this can be avoided for certain
losses (least squares; see Theorem~\ref{c:least-squares} and Lemma \ref{l:gradient-ls}) by relying on
importance sampling.
\end{remark}
\vspace{-1mm}
The proof of Theorem~\ref{t:svrn}, which is given in
Appendix~\ref{a:main-proofs}, relies on a new high-probability
bound for the error of the variance-reduced gradient estimates in the 
large mini-batch regime, measured using the vector norm defined by
the inverse Hessian at the optimum (Lemma~\ref{l:gradient}). Unlike results from
prior work, which hold in expectation, this bound crucially relies on
the iterate being in the local neighborhood. Also, unlike standard
SVRG analysis, we achieve our convergence guarantee for the \emph{last
  iterate} of SVRN's inner loop (as opposed a random or averaged
iterate), which is again enabled by exploiting
local second-order information. 
\begin{figure}
\centering\includegraphics[width = 0.55\textwidth]{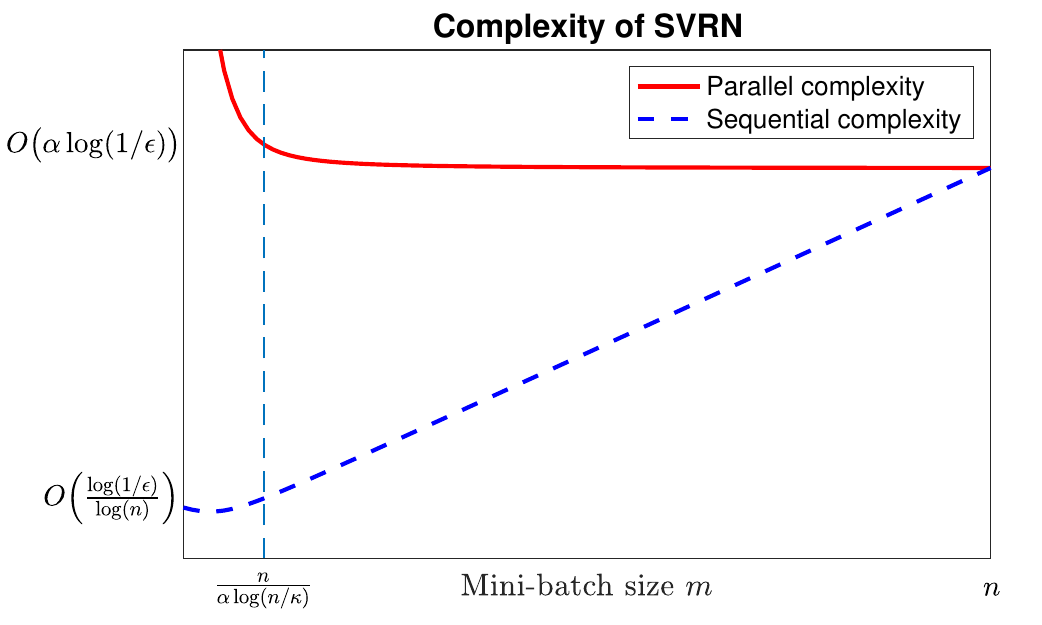} 
 \vspace{-3mm}
  \caption{Illustration of the local convergence complexity analysis
    for SVRN, as a function of the mini-batch size $m$, with the
    number of inner iterations set to $t_{\max}=n/m$. As we decrease
    the mini-batch size from $n$ (standard Stochastic Newton; SN) downto
    $m\approx\frac{n}{\alpha\log(n/\kappa)}$ (optimal SVRN), the sequential
    complexity (number of passes over the data) improves by
    $O(\alpha\log(n))$, while the parallel complexity (number of batch
    gradient queries) remains~optimal.}
  \label{f:svrn_plot}
\end{figure}
\vspace{-1mm}
\paragraph{Discussion.}
For simplicity, let us fix the number of
inner iterations $t_{\max}=n/m$, so that a single outer iteration of SVRN
always takes two passes over the data. Then, we can define the
linear convergence rate after one outer iteration as a function of mini-batch size $m$:
\begin{align*}
\rho_m:=  \Big(1-\frac{1}{2\alpha}\Big)^{n/m} + \tilde O(\kappa/m).
\end{align*}
Let us assume the big data regime, i.e., $n\gg\kappa$. If we
only use full-batch gradients ($m=n$), then the first term in the rate
dominates, and we have $\rho_m \approx 1-\frac1{2\alpha}$, which is
similar to what we would get using standard Stochastic
Newton~\eqref{eq:stochastic-newton}. As we decrease $m$ (and change 
$t_{\max}$ accordingly), the first term in $\rho_m$
decreases, whereas the second term increases. As a result, the overall rate rapidly
improves, reaching its optimal value of $\rho_m=\tilde O(\kappa/n)$ for $m\approx \frac{n}{\alpha\log(n/\kappa)}$.

\paragraph{Complexity analysis.} The complexity analysis given in
Theorem \ref{t:informal} follows directly from the above discussion,
since the sequential complexity (number of data passes needed to improve by factor $\epsilon$) is
given by $O\big(\frac{\log(1/\epsilon)}{\log(1/\rho_m)}\big)$, whereas
the parallel complexity (number of batch gradient queries) is $O\big(t_{\max}\cdot
\frac{\log(1/\epsilon)}{\log(1/\rho_m)}\big)$. In Figure
\ref{f:svrn_plot}, we illustrate how these quantities change as a
function of $m$. In particular, we observe that the batch gradients
essentially stay flat at $O(\alpha\log(1/\epsilon))$ as we decrease $m$,
until reaching $\frac{n}{\alpha\log(n/\kappa)}$. On the other hand, the
data pass complexity decreases linearly with $m$, until it reaches the
optimal value of
$O\big(\frac{\log(1/\epsilon)}{\log(n/\kappa)}\big)$, which,
for sufficiently large $n$, recovers Theorem~\ref{t:informal}.

\section{Globally convergent algorithm}
\label{s:svrn-ha}
We next present a practical stochastic second-order method (see
Algorithm \ref{alg:svrn}, called SVRN-HA) which 
uses SVRN to accelerate its local convergence phase.

The key in
implementing SVRN is that the algorithm is guaranteed to converge with unit step size only once
we reach a local neighborhood of the optimum, and if we have a
sufficiently accurate Hessian estimate. For this reason, we introduce
an initial phase of the algorithm, in which a standard Stochastic Newton method
is ran, using the Armijo line search to select the step size. Once the method
reaches the local convergence neighborhood, as long as the Hessian
estimates are accurate enough, the line search is guaranteed to
return a unit step size. At this point, the algorithm switches to SVRN
and achieves acceleration. 
Finally, to ensure that we reach a sufficiently accurate Hessian estimate, our
Stochastic Newton method should gradually increase the accuracy of the Hessian estimates.

\begin{algorithm}[!t]
\caption{SVRN with Hessian Averaging (SVRN-HA)}
\label{alg:svrn}
\textbf{Input}: iterate $\xbt_0$, gradient batch size $m$, Hessian sample size $k$, and local iterations $t_{\max}$\;
Initialize step size $\eta_{-1} = 0$ and Hessian estimate $\Hbt_{-1} = \zero$\;
\For{ $s = 0,1,2,\ldots$}{
Compute the subsampled Hessian: $\Hbh_s =
\frac1k\sum_{i=1}^k\nabla^2\psi_i(\xbt_s)$,\quad for \quad$\psi_1,...,\psi_k\sim\Dc$\;
Compute the Hessian average: $\Hbt_s = \frac
s{s+1}\Hbt_{s-1}+\frac1{s+1}\Hbh_s$\;
Compute the full gradient: $\gbt_s = \nabla f(\x_s)$\;
\uIf{$\eta_{s-1}< 1$}{
  Compute the descent direction $\tilde\v_s$ by solving: $\Hbt_s\tilde\v_s = -\gbt_s$\;
}\Else{
  Initialize $\x_0 = \xbt_s$\;
  \For{ $t=0,\ldots,t_{\max}-1$}{
    Compute $\gbh_t(\x_t)$ and $\gbh_t(\xbt_s)$,\quad
    for \quad$\gbh_t(\x)=\frac1m\sum_{i=1}^m\nabla\psi_i(\x)$,\quad$\psi_1,...,\psi_m\sim\Dc$\;
    Compute variance-reduced gradient $\bar\g_t=
    \gbh_t(\x_t)-\gbh_t(\xbt_s) + \gbt_s$\;
    Compute the descent direction $\v_t$ by solving: $\Hbt_s\v_t =-\bar\g_t$\;
    Update $\x_{t+1} = \x_t + \v_t$
  }
  Compute the descent direction: $\tilde\v_s=\x_{t_{\max}}-\xbt_s$\;
}
  Compute $\eta_s$ for iterate $\xbt_s$ and direction
  $\tilde\v_s$ using the Armijo condition\;
  Update $\xbt_{s+1} = \xbt_s + \eta_s\tilde\v_s$\;
}\vspace{-1mm}
\end{algorithm}

Based on these insights, we
propose an algorithm called Stochastic Variance-Reduced Newton with Hessian Averaging
(SVRN-HA). In the initial phase, this algorithm is a variant of
Subsampled Newton, based on a method proposed by
\cite{hessian-averaging}, where, at each iteration, we construct a
subsampled Hessian estimate based on a fixed sample size $k$. To
increase the accuracy over time, all past Hessian estimates are
averaged together, and the result is used to precondition the full
gradient. At each iteration, we check whether the last line
search returned a unit step size. If yes, then we start running SVRN
with local iterations $t_{\max}=\lfloor\log_2(n/d)\rfloor$ and gradient batch size
$m=\lfloor n/\log_2(n/d)\rfloor$, where $n$ is the number of data points and $d$ is the
dimension. This is motivated by our theory (see
discussion below Theorem \ref{t:svrn}), using $d$ as a proxy for the condition
number $\kappa$. This choice has proven effective in all of the
evaluated datasets, which indicates that the theoretical dependence of
the mini-batch size on the condition number is likely quite pessimistic.

In the following result, we establish global convergence of SVRN-HA,
by showing that the global phase of this method will not only reach any local
neighborhood, but also that the Hessian estimate will get
progressively more accurate, eventually reaching the desired
approximation accuracy.
\vspace{-1mm}
\begin{theorem}\label{t:global}
Let $f$ be as in Theorem \ref{t:svrn}. For any neighborhood $U$ around
the optimum, Algorithm \ref{alg:svrn}
will almost surely reach a point where: (a) $\xbt_s$ belongs to the
neighborhood $U$, 
  and (b) the Hessian estimate $\Hbt_s$ satisfies the condition in Theorem \ref{t:svrn}.
At this point, the line search will return $\eta_s = 1$.
\end{theorem}

\section{Experiments}
\label{s:experiments}

We next demonstrate numerically that SVRN can be effectively used to
accelerate stochastic Newton methods in practice. We also show how
variance reduction can be incorporated into a globally convergent Subsampled
Newton method in a way that is robust to hyperparameters and
preserves its scalability thanks to large-batch operations.\footnote{The code is available at
  \url{https://github.com/svrnewton/svrn}.}

\subsection{Logistic regression experiment}
In this
section, we present numerical experiments for solving a  regularized logistic
loss minimization task.
For an $n\times
d$ data matrix $\A$ with rows $\a_i^\top$, an $n$-dimensional target
vector $\y$ (with $\pm1$ entries $y_i$) and a regularization parameter
$\gamma$, our task is to minimize:
\begin{align}
f(\x) =
  \frac1n\sum_{i=1}^n\log(1+\ee^{-y_i\a_i^\top\x}) + \frac\gamma2\|\x\|^2.\label{eq:lr}
\end{align}
As a dataset, we used the Extended
MNIST dataset of handwritten digits \cite[EMNIST]{cohen2017emnist} with
$n=500$k datapoints,  transformed using a
 random features map (with dimension $d=1000$). Experimental
details, as well as further results on the CIFAR-10 dataset
and several
synthetic data matrices, are presented in Appendix
\ref{a:experiments}.

\ifisarxiv\begin{figure} [!th]
\centering   
\subfigure[Convergence ($d=1000$)]{\label{fig:1000-iters}\includegraphics[width=6.5cm]{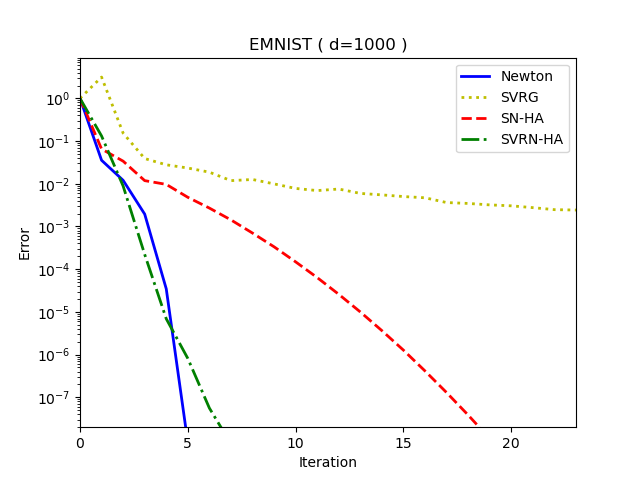}}
\subfigure[Runtime ($d=1000$)]{\label{fig:1000-time}\includegraphics[width=6.5cm]{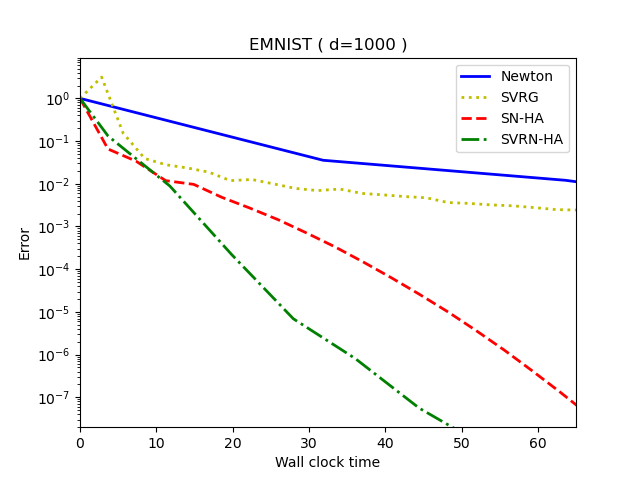}}
\vspace{-2mm}\caption{Convergence and runtime comparison of SVRN-HA against three
  baselines: classical Newton, SVRG (after parameter tuning), and
  Subsample Newton with Hessian Averaging (SN-HA), i.e., the initial
  phase of SVRN-HA ran without variance reduction all the way
  through.}\label{fig:plots}
\vspace{-1mm}
\end{figure}
\else
\begin{figure*} 
  \vspace{-3mm}
  \centering     
\subfigure[Convergence]{\label{fig:1000-iters}\includegraphics[width=.4\textwidth]{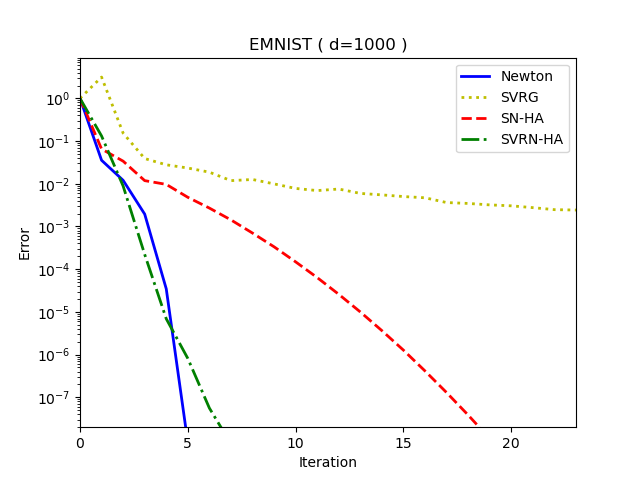}}
\subfigure[Runtime (seconds)]{\label{fig:1000-time}\includegraphics[width=.4\textwidth]{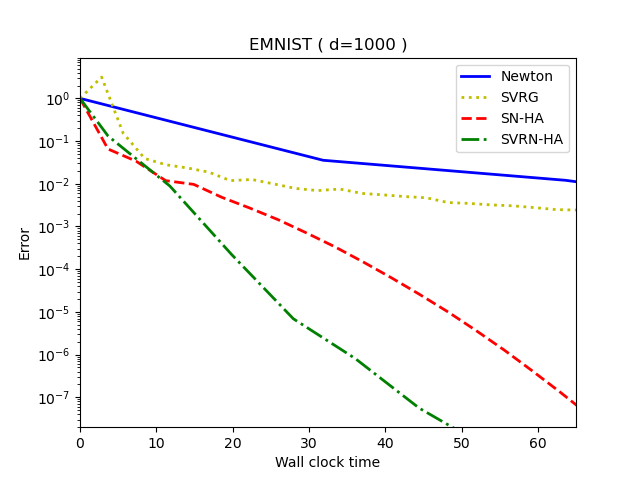}}
\vspace{-1mm}
\caption{Convergence and runtime comparison of SVRN-HA on the EMNIST
  dataset against three 
  baselines: classical Newton, SVRG (after parameter tuning), and
  Subsampled Newton with Hessian Averaging (SN-HA), i.e., the global
  phase of Algorithm \ref{alg:svrn}, ran without switching to SVRN. Further results on the CIFAR-10 dataset are in Appendix \ref{a:experiments}.}\label{fig:plots}
\vspace{-1mm}
\end{figure*}
\fi

In Figure \ref{fig:plots}, we compared SVRN-HA to three baselines which are most
directly comparable: (1) the classical Newton's method; (2) SVRG with
the step size and number of inner iterations tuned for best wall-clock time; and (3) Subsampled Newton
with Hessian Averaging (SN-HA), i.e., the method we use in the
global phase of Algorithm \ref{alg:svrn} (without the SVRN phase). All
of the convergence plots are averaged 
over 10 runs. For both SVRN-HA and SN-HA we use
Hessian sample size $k=4d$.

From Figure \ref{fig:1000-iters}, we conclude
that as soon as SVRN-HA exits the initial phase of the optimization,
it accelerates dramatically, to the point where it nearly matches the rate
of classical Newton. This acceleration corresponds to the improvement
in sequential complexity from $O(\alpha\log(1/\epsilon))$ for
Stochastic Newton to $O(\frac{\log(1/\epsilon)}{\log(n)})$ for
SVRN. In all our experiments, the transition to the  SVRN phase
in SVRN-HA occurred very quickly, generally within 1-2 iterations,
which indicates that the method easily reaches local convergence.
Finally, we observed that the convergence of SVRG is initially quite fast, but over time,
it stabilizes at a slower rate, indicating that the Hessian
information plays a significant role in the performance of SVRN-HA.

In Figure \ref{fig:1000-time}, we plot the
wall clock time of the algorithms. Here, SVRN-HA also performs better
than all of the baselines, despite some additional per-iteration overhead.
We expect that
this can be further optimized. Finally, we note that Newton's method is
drastically slower than all other methods due to the high cost of solving a
large linear system, and the per-iteration time of SVRG is substantially slowed by its
sequential nature.

\subsection{Further investigations on a least squares task}
\label{s:experiments-ls}

We next study the setting of least squares
regression \eqref{eq:ls-first} to analyze the trade-offs in convergence for different
implementations of SVRN, as we vary the gradient and Hessian
estimation schemes. We evaluated the algorithms on synthetic data
matrices, as defined in Appendix \ref{a:experiments}.

\paragraph{Communication cost of gradient resampling.}
Our theoretical analysis requires that for each small step of SVRN, a
fresh sample of components $\psi_i$ is used to compute the gradient
estimates. However, in Lemma~\ref{l:gradient} we showed that, after
variance reduction, the gradient estimates are accurate with high 
probability, which suggests that we might be able to reuse previously
sampled components. While this technically does not improve the number
of required gradient queries, it can substantially reduce the
communication cost for some practical implementations.

As an example,
let us consider the setting where each full/mini-batch gradient
computation requires reading the corresponding data chunk from the
server onto the computing core. Then, the theoretical version of SVRN
requires reading the entire dataset roughly $2$ times ($2n$ data
points) in the course of one stage (outer iteration): $n$ data points for computing the full gradient, and then $t_{\max}\cdot m = (n/m)\cdot m = n$ data points for all of the mini-batch steps together. On the other
hand, if we were to reuse the same mini-batch for all of the inner
iterations in one stage, then we require reading the dataset only
$1+o(1)$ times.

In
Figure~\ref{fig:ls-sampling}, we investigate how much the convergence rate
of SVRN-HA is affected by the frequency of component resampling for
the gradient estimates. Recall that in all our experiments, we use a
gradient sample size of $m=\lfloor n/\log_2(n/d)\rfloor$. We consider the following variants:
\vspace{-3mm}\begin{enumerate}
  \item \underline{Sampling once}: an extreme policy of sampling one set
    of components and reusing them for all gradient estimates;\vspace{-2mm}
  \item \underline{Sampling per stage}: an intermediate policy of
    resampling the components after every full gradient computation.\vspace{-2mm}
  \item \underline{Sampling per step}: the policy which is used in our
    theory, i.e., resampling the gradients at every step of the inner
    loop of the algorithm.
  \end{enumerate}
  From Figure \ref{fig:ls-sampling} we conclude that, while all three
  variants of SVRN-HA converge and are competitive with SN-HA, the
  extreme policy of sampling once leads to a substantial
  degradation in convergence rate, whereas sampling per stage and
  sampling per step perform very similarly. Thus, our overall
  recommendation is to resample the components at every stage of
  SVRN-HA, but reuse the sample for the small steps of the algorithm
  (this is what we used for the EMNIST and CIFAR-10 experiments).

\begin{figure} 
\centering     
\subfigure[Comparison of three variants of SVRN-HA (alongside SN-HA), depending on how
  frequently we resample data points used to compute the gradient estimate. We consider three
  variants of SVRN-HA: (1) sampling once for the entire optimization,
  (2) sampling once for each full gradient stage (per stage),
  (3) sampling in each small step (per step).] {\label{fig:ls-sampling}
\includegraphics[width=0.45\textwidth]{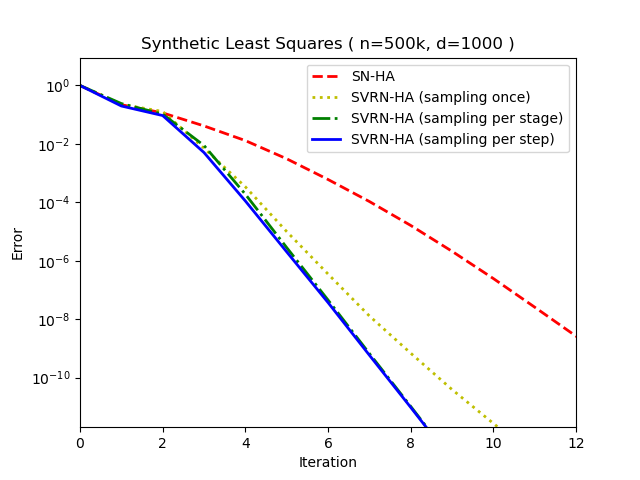} 
}%
\hspace{5mm}\subfigure[Comparison of SVRN-HA (alongside SN-HA) against
Subsampled Newton with Gradient Subsampling (SNGS-HA), which is
implemented exactly like SVRN-HA except without the
variance-reducing correction.]{\label{fig:ls-vr}\includegraphics[width=.45\textwidth]{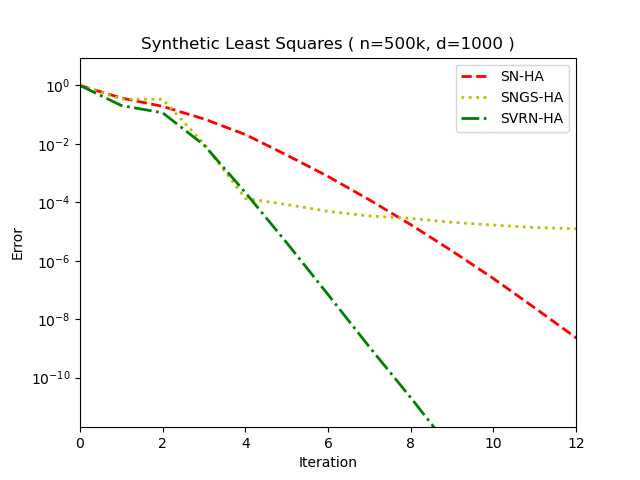}}
\caption{How different types of gradient estimation
  affect the convergence properties of SVRN.}\label{fig:ls-gradient}
\end{figure}

\paragraph{Effect of variance reduction.}
We next investigate the effect of variance reduction on the
convergence rate of SVRN. While gradient subsampling has been proposed
by many works in the literature on Subsampled Newton \cite{roosta2019sub}, these works have shown that the gradient sample size must be
gradually increased to retain fast local convergence (which means that
after a few iterations, we must use the full gradient). On the other
hand, in SVRN, instead of increasing the gradient sample size, we use
variance reduction with a fixed sample size, which allows us to
retain the accelerated convergence indefinitely.

To illustrate this point, in Figure \ref{fig:ls-vr}
we plot how the convergence behavior of our algorithm changes if
we take variance reduction out of it. The resulting method is called
Subsampled Newton with Gradient Subsampling (SNGS-HA). For this experiment, we resample the gradient
estimate at every small step (for both SNGS-HA and SVRN-HA). For the
sake of direct comparison, all of the other parameters are retained from
SVRN-HA. In particular, one iteration of SNGS-HA corresponds to
$\lfloor\log_2(n/d)\rfloor$ steps using resampled gradients, and Hessian
averaging occurs once every such iteration. As expected, we observe
that, while initially converging at a fast rate, eventually SNGS-HA
reaches a point where the subsampled gradient estimates are not
sufficiently accurate, resulting in a sudden dramatic drop-off in the
convergence rate, to the point where the method virtually stops
converging altogether. On the other hand, SVRN-HA
continues to converge at the same fast rate throughout the
optimization procedure without any reduction in performance. This
indicates that variance reduction does improve the accuracy of
gradient~estimates, especially when our goal is to converge to a
high-precision solution.

  \begin{figure} 
\centering   
\subfigure[Convergence comparison of SVRN and SN using fixed Hessian
  estimates (i.e., without Hessian averaging). Here, $h$ denotes the
  number of Hessian samples used to generate the estimate.]{\label{fig:ls-noavg}
\includegraphics[width=0.45\textwidth]{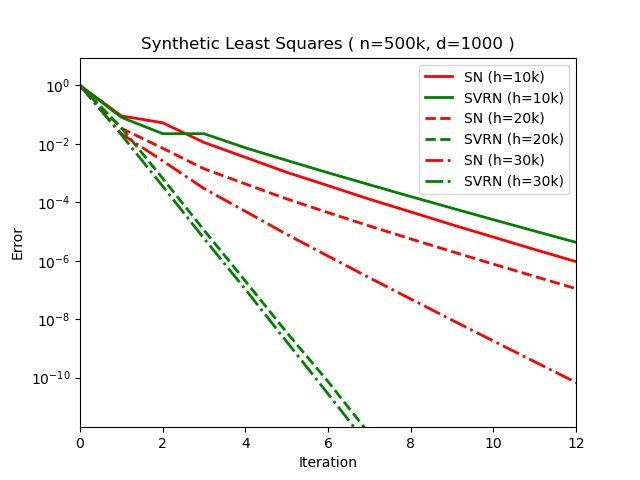}
}\hspace{5mm}%
\subfigure[Convergence comparison of SVRN-HA and SN-HA, with and without preconditioning
using a Randomized Hadamard Transform (RHT), for a high-coherence least squares
dataset.]{\label{fig:ls-coherence}\includegraphics[width=.45\textwidth]{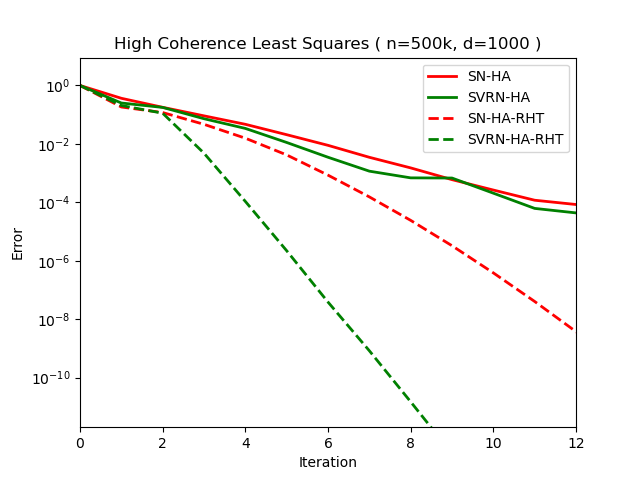}}
\caption{How Hessian sample size and data coherence affect
the convergence properties of SVRN.}\label{fig:ls-other}
\end{figure}

\paragraph{Effect of Hessian accuracy.} In our experiments, for
both SVRN and SN, we used Hessian averaging \cite{hessian-averaging}
to construct the Hessian estimates. This approach is
desirable in practice, since it gradually increases
the accuracy of the Hessian estimate as we progress in the optimization.
As a result, it is more robust to the Hessian sample size and we are
guaranteed to reach sufficient accuracy for SVRN to work
well. In the following experiment, we take Hessian averaging out of
the algorithms to provide a better sense of how the
performance of SVRN and SN depends on the accuracy of the provided
Hessian estimate. For simplicity, we focus here on least
squares, where the Hessian is the same everywhere, so we can simply
construct an initial Hessian estimate and then use it throughout the
optimization. However, our insights apply more broadly to local
convergence for general convex objectives. In
Figure~\ref{fig:ls-noavg}, we plot the performance of the algorithms
as 
we vary the accuracy of the subsampled Hessian estimates, where $h$
denotes the number of samples used to construct the estimate. In all
the results, we keep the gradient sample size and local
steps in SVRN fixed as before.

Remarkably, the performance of SVRN
is affected by the Hessian accuracy very differently than SN. We
observe that SVRN requires a certain level of Hessian accuracy to provide any
acceleration over SN. As soon as this level of Hessian accuracy is
reached (by increasing the Hessian sample size $h$), the peformance of
SVRN quickly jumps to the fast convergence 
we observed in the other experiments. Further increasing the accuracy
no longer appears to affect the convergence rate of SVRN. This is
in contrast to SN, whose convergence slowly improves as we increase
the Hessian sample size. This intriguing phenomenon is actually fully
predicted by our theory for SVRN (together with prior convergence
analysis for SN). Our convergence result (Theorem~\ref{t:svrn}) requires a sufficiently
accurate Hessian inverse estimate for SVRN to work with a unit step
size (which is what is used in SVRN-HA), but the actual rate of
convergence is independent of the Hessian accuracy (only the required number of
small steps is affected). We conclude that SVRN is more desirable than SN when we
have a small budget for Hessian samples.

\paragraph{Effect of high coherence.}
We next analyze the performance of SVRN and SN on a slightly modified
least squares task. For this experiment, we modify the data matrix
$\A$, by multiplying the $i$th row by $1/\sqrt{g_i}$ for each $i$, where $g_i$ is an independent random variable
distributed according to the Gamma distrution with shape $2$ and scale
$1/2$. This is a standard transformation designed to produce a matrix
with many rows having a high leverage score. Recall that the leverage
score of the $i$th row of $\A$ is defined as
$\ell_i=\a_i^\top(\A^\top\A)^{-1}\a_i$, see Appendix \ref{a:ls}. This
can be viewed as affecting the component-wise smoothness of the
objective, which hinders subsampling-based estimators of the Hessian
and the~gradient.

In Figure \ref{fig:ls-coherence}, we illustrate how
the performance of SVRN-HA and SN-HA degrades for the high-coherence least squares task,
and we also show how this can be addressed by relying on the ideas
developed in Section~\ref{s:ls} (and further discussed in Appendix \ref{a:ls}). First, notice that not only is the
convergence rate of both SVRN-HA and SN-HA worse on the high-coherence
dataset than on the previous least squares examples (e.g., compare
with Figure \ref{fig:ls-vr}), but also, the acceleration coming from
variance reduction is drastically reduced to the point of being
negligible. The former effect is primarily caused by the fact that
uniform Hessian subsampling is much less effective at producing
accurate approximations for high-coherence matrices, and this affects
both algorithms similarly (we note that one could construct an even more
highly coherent matrix, for which these methods would essentially stop
converging altogether). The latter effect is the consequence of the fact that gradient 
subsampling is also adversely affected by high coherence, so it becomes
nearly impossible to produce gradient estimates with uniform sampling
that would lead to an accelerated rate, even with variance
reduction. This corresponds to the regime of $\kappa\geq n$ in our
theory.

Fortunately, for least squares regression, this phenomenon can be
addressed easily. As outlined in Appendix~\ref{a:ls}, we can use one
of two strategies: (1) use importance sampling proportional to the
leverage scores of $\A$ for both the Hessian and gradient estimates; or
(2) precondition the problem using the Randomized Hadamard Transform
(RHT)  to uniformize all the leverage scores, and then use uniform
subsampling. Both of these methods require roughly $O(nd\log n)$
preprocessing cost and eliminate dependence on the condition
number for both SVRN-HA and SN-HA. The latter strategy is somewhat
more straightforward since it does not require modifying the
optimization algorithms, and we apply it here for our high coherence
least squares task: we let SVRN-HA-RHT and SN-HA-RHT denote the
two optimization algorithms ran after applying the RHT preconditioning to the problem. Note that this not
only improves the convergence rate of both methods but also brings
back the accelerated rate enjoyed by SVRN-HA in the previous
experiments. In fact, our least squares results (Theorem~\ref{c:least-squares} and Lemma \ref{l:gradient-ls}) can be directly
applied to SVRN-HA-RHT, so this method and its accelerated convergence
rate of $\tilde O(d/n)$ is
provably unaffected by any high-coherence matrices.

\section{Conclusions}

We propose and analyze Stochastic Variance-Reduced Newton (SVRN), a provably
effective strategy of incorporating variance reduction into popular
stochastic Newton methods for solving finite-sum minimization
tasks. We show that SVRN improves the local convergence complexity of Subsampled Newton (per data pass) from $O(\alpha\log(1/\epsilon))$ to $O\big(\frac{\log(1/\epsilon)}{\log(n)}\big)$, while
retaining all the benefits of second-order optimization, such
as a simple unit step size and easily scalable large-batch operations.

\bibliography{../pap}
\bibliographystyle{plain}

\appendix

 \section{Proofs for local convergence analysis of SVRN}
 \label{a:main-proofs}

In this section, we provide the proofs of our main technical results, i.e., local convergence
analysis for SVRN. First, we prove the result for the general case
(Theorem \ref{t:svrn}), then we prove the result for least squares
(Theorem \ref{c:least-squares}).

\subsection{Preliminaries}

First, let us recall the formal definitions of the standard Hessian
regularity assumptions used in Theorem \ref{t:svrn}. For all our
results, it is sufficient that the function $f$ satisfies either one
of these assumptions.

\begin{assumption}\label{a:lipschitz}
  Function $f:\R^d\rightarrow \R$ has Lipschitz continuous Hessian
  with constant $L$, i.e., $\|\nabla^2 f(\x)-\nabla^2 f(\x')\|\leq
L\,\|\x-\x'\|$ for all $\x,\x'\in\R^d$.
\end{assumption}
\begin{assumption}\label{a:self-concordant}
  Function $f:\R^d\rightarrow \R$ is self-concordant, i.e., for all
  $\x,\x'\in\R^d$, the function $\phi(t)=f(\x+t\x')$
  satisfies: $|\phi'''(t)|\leq 2(\phi''(t))^{3/2}$.
\end{assumption}

In the proof, we use the following version of Bernstein's
concentration inequality for random vectors \cite[Corollary 4.1]{minsker2017some}.
\begin{lemma}\label{l:vector-bernstein}
Let $\v_1,...,\v_m\in\R^d$ be independent random vectors such
that $\E[\v_i] = \zero$ and $\|\v_i\|\leq R$ almost surely. Denote
$\sigma^2:=\sum_{i=1}^m\E\,\|\v_i\|^2$. Then, for all $t^2\geq
\sigma^2+tR/3$, we have
\begin{align*}
  \Pr\bigg\{\Big\|\sum_{i=1}^m\v_i\Big\| > t\bigg\}
  \leq 28\exp\Big(-\frac{t^2/2}{\sigma^2+tR/3}\Big).
\end{align*}
\end{lemma}
We also use the following lemma to convert from convergence in the
norm, $\|\x-\x^*\|_{\H}$, to convergence in excess loss,
$f(\x)-f(\x^*)$, in the neighborhood around the optimum $\x^*$. The
proof of this lemma, given in Appendix~\ref{a:conversion}, uses Quadratic Taylor's Theorem.
\begin{lemma}\label{l:conversion}
If $f$ satisfies Assumption \ref{a:convex} and either Assumption \ref{a:lipschitz} or
\ref{a:self-concordant}, then for any $\epsilon \in[0,1]$ and $\x\in
U_f(\epsilon_l)$, we have:
  \vspace{-1mm}
  \begin{align*}
    \nabla^2f(\x)\approx_{\epsilon_l}\nabla^2f(\x^*)\qquad\text{and}\qquad
    f(\x)-f(\x^*)\approx_{\epsilon_l} \frac12\|\x-\x^*\|_{\nabla^2f(\x^*)}^2.
  \end{align*}
\end{lemma}

\subsection{Proof of Theorem \ref{t:svrn}}
To simplify the notation, we will drop the subscript $s$, so that
$\xbt = \xbt_s$ and $\gbt=\gbt_s$. Also, let us define $\gbh(\x) = 
\frac1m\sum_{i=1}^m\nabla\psi_i(\x)$. We use 
$\g(\x)=\nabla f(\x)$, $\g_t=\g(\x_t)$, $\H_t=\nabla^2f(\x_t)$, $\H=\nabla^2f(\x^*)$, $\gbh_t=\gbh(\x_t)$, and $\bar\g_t=\gbh_t-\gbh(\xbt)+\gbt$ as
shorthands. Also, we will use $\Delta_t=\x_t-\x^*$. We start by
splitting up the error bound into two terms: the first one is an error
term that would arise if we were using the exact gradient $\g_t$
instead of the gradient estimate $\bar\g_t$; and the second term
addresses the error coming from the noise in the gradient estimate.
Initially, we
use the error $\|\Delta_t\|_{\H}$ to analyze the convergence
rate in one step of the procedure, where recall that
$\|\v\|_{\M}=\sqrt{\v^\top\M\v}$. We then convert that to get 
convergence in function~value via Lemma~\ref{l:conversion}. 

We first address the assumption that $\Hbt$ is an
$\alpha$-approximation of $\H_t$, as defined by the condition \eqref{eq:alpha}. Note that, via
Lemma~\ref{l:conversion}, for any $\x_t\in U_f(\epsilon_l)$ we have that $\H_t\approx_{\epsilon_l}\H$,
which 
for a sufficiently small $\epsilon_l$ implies that $\H_t$ is a
$1.1$-approximation of $\H$ in the sense of \eqref{eq:alpha}. This, in turn implies that $\Hbt$ is a
$1.1\alpha$-approximation of $\H$, because $\Hbt\preceq
\sqrt\alpha\H_t\preceq \sqrt{1.1\alpha}\H$ (the other direction is
analogous). For the sake of simplicity, we will 
replace $1.1\alpha$ with $\alpha$ and say that $\Hbt$ is an
$\alpha$-approximation of $\H$ (this can be easily accounted for by
adjusting the constants at the end).

Now, suppose that after $t$ inner iterations, we get $\x_t\in
U_f(\epsilon_l)$ satisfying $\|\Delta_t\|_{\H}\leq\|\Delta_0\|_{\H}$. Our decomposition of the error
into two terms proceeds as follows, where we use $\pbt_t = \Hbt^{-1}\bar\g_t$:
\begin{align}
  \|\Delta_{t+1}\|_{\H}
  &=\|(\x_t  - \eta\pbt_t) - \x^*\|_{\H}\nonumber
    \\
  &= \|\Delta_t- \eta\Hbt^{-1}\g_t  + \eta\Hbt^{-1}\g_t -
    \eta\Hbt^{-1}\bar\g_t\|_{\H}\nonumber
  \\
  &\leq \|\Delta_t -\eta\Hbt^{-1}\g_t\|_{\H}+\eta\|\Hbt^{-1}(\g_t-\bar\g_t)\|_{\H}
    \label{eq:decomposition}
\end{align}
To bound the second term in \eqref{eq:decomposition},
we first observe that $\Hbt^{-1}\preceq \sqrt\alpha\H^{-1}$, 
which in turn yields
$\H^{1/2}\Hbt^{-1}\H^{1/2}\preceq \sqrt\alpha\,\I$. Thus, we can
write $\|\H^{1/2}\Hbt^{-1}\H^{1/2}\|\leq \sqrt\alpha$ and we get:
\begin{align*}
  \|\Hbt^{-1}(\g_t-\bar\g_t)\|_{\H}
  & =
    \|\H^{1/2}\Hbt^{-1}\H^{1/2}\H^{-1/2}(\g_t-\bar\g_t)\|
  \\
  &\leq  \|\H^{1/2}\Hbt^{-1}\H^{1/2}\|\cdot
    \|\H^{-1/2}(\g_t-\bar\g_t)\|
  \\
  &\leq \sqrt{\ch}\cdot \|\g_t - \bar\g_t\|_{\H^{-1}}
\end{align*}
We now break $\|\g_t - \bar\g_t\|_{\H^{-1}}$
 down into two parts, introducing $\gbh(\x^*)$ and separating
$\gbh(\xbt)$ from $\gbh_t$: 
\begin{align*}
  \|\g_t - \bar\g_t\|_{\H^{-1}}
  &= \|
  \g_t
   - (\gbh_t - \gbh(\xbt) + \gbt)\|_{\H^{-1}}
  \\
  &\leq \|\g_t - (\gbh_t - \gbh(\x^*)) \|_{\H^{-1}} +
   \|\gbt - (\gbh(\xbt)-\gbh(\x^*)) \|_{\H^{-1}}.
\end{align*}
We
bound the above two terms using the following lemma, which gives a new
high-probability error bound for the variance reduced gradient
estimates in the large mini-batch regime which, unlike results from
prior work
that hold in expectation, crucially relies on the iterate being in the
local neighborhood $U_f(1)$ around the optimum $\x^*$.
\begin{lemma}\label{l:gradient}
There is an absolute constant $C>0$
  such that for any $\x\in U_f(1)$, letting
$\H=\nabla^2 f(\x^*)$, the
  gradient estimate $\gbh(\x) = \frac1m\sum_{i=1}^m\nabla\psi_i(\x)$ using
  $m\geq \kappa\log(1/\delta)$ samples, with probability $1-\delta$ satisfies: 
  \begin{align*}
\big\|\gbh(\x)-\gbh(\x^*)-\nabla f(\x)\big\|_{\H^{-1}}^2\!\leq
    C\log(1/\delta) \frac{\kappa}{m}\|\x-\x^*\|_{\H}^2.
  \end{align*}
\end{lemma}
\begin{proof}
We will apply Bernstein's concentration inequality for random vectors
(Lemma \ref{l:vector-bernstein}) to $\v_i = 
\nabla\psi_i(\x)-\nabla\psi_i(\x^*) - \nabla f(\x)$.
First, observe that $\E\,\nabla\psi_i(\x)=\nabla f(\x)$ and $\E\,\nabla\psi_i(\x^*)=\nabla
f(\x^*)=\zero$, so in particular, $\E[\v_i]=\zero$.

In the next step, we will use
the fact that for any $\lambda$-smooth function $g$, we have
$\|\nabla g(\x)\|^2\leq 2\lambda\cdot(g(\x)-\min_{\x'}g(\x'))$, which
follows because:
\begin{align*}
  \min_{\x'}g(\x')
  &\leq g\big(\x-\tfrac1{\lambda}\nabla g(\x)\big)
  \\
  &\leq g(\x) -
  \frac1\lambda\|\nabla g(\x)\|^2 + \frac\lambda2\,\frac1{\lambda^2}\|\nabla
    g(\x)\|^2 \\
  &= g(\x)-\frac1{2\lambda}\|\nabla g(\x)\|^2.
\end{align*}
We will use this fact once on $f$, and also a second time, on the
function
$g(\x)=\psi_i(\x)-\psi_i(\x^*)-(\x-\x^*)^\top\nabla\psi(\x^*)$, which
is $\lambda$-smooth because $\psi_i$ is $\lambda$-smooth,
observing that $\nabla g(\x) = \nabla \psi_i(\x)-\nabla\psi_i(\x^*)$
and that $\min_{\x'}g(\x') = g(\x^*)=0$. Thus, we have
\begin{align*}
  \|\v_i\|^2
  &\leq 2\|\nabla\psi_i(\x)-\nabla\psi_i(\x^*)\|^2 + 2\|\nabla
    f(\x)\|^2
  \\
  &\leq 4\lambda\cdot\big(\psi_i(\x) - \psi_i(\x^*) -
    (\x-\x^*)^\top\nabla\psi(\x^*)\big)
+ 4\lambda\cdot \big(f(\x)-f(\x^*)\big)
  \\
  &\leq 2\lambda^2\|\x-\x^*\|^2 + 2\lambda^2\|\x-\x^*\|^2
= 4\lambda^2\|\x-\x^*\|^2,
\end{align*}
where in the last step we used again that $\psi_i$ and $f$ are $\lambda$-smooth.
To bound the expectation $\E\,\|\v_i\|^2$, we use the intermediate
inequality from the above derivation, obtaining:
\begin{align*}
  \E\,\|\v_i\|^2
  &=\E\big[\|\nabla\psi_i(\x)-\nabla\psi_i(\x^*)\|^2\big] -
    2\,\E\big[\nabla\psi_i(\x)-\nabla\psi_i(\x^*)\big]^\top\nabla f(\x) +
    \|\nabla f(\x)\|^2
  \\
  &=\E\big[\|\nabla\psi_i(\x)-\nabla\psi_i(\x^*)\|^2\big] - \|\nabla
    f(\x)\|^2
    \\
  &\leq \E\big[ \|\nabla\psi_i(\x)-\nabla\psi_i(\x^*)\|^2\big]
  \\
  &\leq \E\Big[2\lambda\cdot\big(\psi_i(\x)-\psi_i(\x^*) -
    (\x-\x^*)^\top\nabla\psi_i(\x^*)\big)\Big]
  \\
  &=2\lambda\cdot\big(f(\x) - f(\x^*) - (\x-\x^*)\nabla f(\x^*)\big)
  \\
  &=2\lambda \cdot \big(f(\x)-f(\x^*)\big).    
\end{align*}
We now use the assumption that $\x\in U_f(1)$, which implies
via Lemma \ref{l:conversion} that $f(\x)-f(\x^*)\leq
2\cdot\frac12\|\x-\x^*\|_{\H}^2=\|\x-\x^*\|_{\H}^2$. 
Thus, we can use Lemma \ref{l:vector-bernstein} with $R = 
2\lambda\|\x-\x^*\|$ and $\sigma^2 = 2m\lambda\|\x-\x^*\|_{\H}^2$, as well as $\mu$-strong convexity of $f$,  obtaining
that, for some absolute constant $C$, with
probability $1-\delta$, we have:
\begin{align*}
  \|\gbh(\x) - \gbh(\x^*) - \nabla f(\x)\|_{\H^{-1}}^2
  &\leq \frac1{\mu}\Big\|\frac1m\sum_{i=1}^m\v_i\Big\|^2
  \\
  &\leq \frac {C}{\mu}\Big(\frac{\sigma^2\log(1/\delta)}{m^2} +
    \frac{R^2\log^2(1/\delta)}{m^2}\Big)
  \\
  &\leq \frac C\mu\bigg(\frac{2\lambda\|\x-\x^*\|_{\H}^2\log(1/\delta)}{m}
    + \frac{4\lambda^2\|\x-\x^*\|^2\log^2(1/\delta)}{m^2}\bigg)
  \\
  &\leq 4C\bigg(\frac{\kappa\log(1/\delta)}{m} +
    \frac{\kappa^2\log^2(1/\delta)}{m^2}\bigg)\cdot
    \|\x-\x^*\|_{\H}^2
  \\
  &\leq 8C\log(1/\delta)\cdot \frac{\kappa}{m}\,\|\x-\x^*\|_{\H}^2,
\end{align*}
where in the last step we used that $m\geq \kappa\log(1/\delta)$.
\end{proof}

Letting $\epsilon_g\! =\!
\sqrt{2C\log(t/\delta)\kappa/m}$, Lemma \ref{l:gradient} implies that with probability $1-\delta/t^2$,
\begin{align*}
\|\g_t - (\gbh_t - \gbh(\xbt) + \gbt)\|_{\H^{-1}}
&\leq \epsilon_g\big(\|\Delta_t\|_{\H} +
\|\Delta_0\|_{\H}\big)\leq 2\epsilon_g\|\Delta_0\|_{\H}.
\end{align*}

Finally, we return to the first term in
\eqref{eq:decomposition}, i.e.,
$\|\Delta_t-\eta\Hbt^{-1}\g_t\|_{\H}$. To control this term we
introduce the following lemma which is potentially of independent
interest to the local convergence analysis of Newton-type methods.

\begin{lemma}\label{l:newton}
Suppose that $f$ satisfies Assumption \ref{a:convex} and either one of
the Assumptions \ref{a:lipschitz} or \ref{a:self-concordant}, and take
any $\x\in U_f(\epsilon_l)$ (see
Definition~\ref{d:neighborhood}) for $\epsilon_l\leq 1/c\alpha$
for a sufficiently large absolute constant $c>0$. Let $\H=\nabla^2f(\x^*)$ and consider a pd matrix $\tilde \H$ that
satisfies $\frac1{\sqrt\alpha}\,\H\preceq\tilde\H\preceq\sqrt\alpha\,\H$.
Then, for $\eta:= \min\{\sqrt{2/\alpha},1\}$, we have:
\begin{align*}
  \|\x - \eta\Hbt^{-1}\nabla f(\x) - \x^* \|_{\H}
  \leq \Big(1-\frac1{1.9\alpha}\Big)\|\x-\x^*\|_{\H}.
  \end{align*}
\end{lemma}
\begin{proof}
  Let $\Delta_0:=\x-\x^*$ and $\Delta_1:=\x-\eta\tilde\H^{-1}\nabla
  f(\x)-\x^*$. Using that $\nabla f(\x^*)=\zero$, we have: 
  \begin{align*}
    \Delta_1
    &= \Delta_0 - \eta\tilde\H^{-1}\nabla f(\x)
      \\
    &= \Delta_0 - \eta\tilde\H^{-1}(\nabla f(\x) - \nabla f(\x^*))
    \\
    & = \Delta_0
      -\eta\tilde\H^{-1}\int_0^1\nabla^2f(\x^*+\theta\Delta_0)\Delta_0
      d\theta
    \\
    &=(\I - \eta\tilde\H^{-1}\bar\H)\Delta_0,
  \end{align*}
  where we defined
  $\bar\H:=\int_0^1\nabla^2f(\x^*+\theta\Delta_0)d\theta$. It follows
  that we can bound the norm of $\Delta_1$ using a norm defined by the
  matrix $\bar\H$:
  \begin{align*}
    \|\Delta_1\|_{\bar\H}
    &=
      \|\bar\H^{1/2}(\I-\eta\tilde\H^{-1}\bar\H)\Delta_0\|
    \\
    &=\|(\I-\eta\H^{1/2}\tilde\H^{-1}\bar\H^{1/2})\bar\H^{1/2}\Delta_0\|
    \\
    &\leq\|\I-\eta\bar\H^{1/2}\tilde\H^{-1}\bar\H^{1/2}\|\cdot\|\Delta_0\|_{\bar\H}.
  \end{align*}
  Observe that for any $\theta\in[0,1]$, the vector $\x^*+\theta\Delta_0$
  belongs to $U_f(\epsilon_l)$, which via Lemma \ref{l:conversion} implies that
  \begin{align*}
    \nabla^2f(\x^*+\theta\Delta_0)\approx_{\epsilon_l}\H\quad\forall\theta\in[0,1].
  \end{align*}
where $\H=\nabla^2f(\x^*)$.  In particular, this means that $\bar\H\approx_{\epsilon_l}\H$,
  which, combined with the $\alpha$-approximation property of
  $\tilde\H$, 
  allows us to write the following:
  \begin{align*}
    \tilde\H^{-1}\preceq \sqrt\alpha\H^{-1}\preceq
    \sqrt\alpha(1+\epsilon_l)\bar\H^{-1}\quad\text{and}\quad
    \tilde\H^{-1}\succeq \frac1{\sqrt\alpha}\H^{-1}\succeq \frac{1-\epsilon_l}{\sqrt\alpha}\bar\H^{-1}.
  \end{align*}
  Putting these inequalities together, we obtain that:
  \begin{align*}
    \eta\frac{1-\epsilon_l}{\sqrt\alpha}\I
    &\preceq\eta\bar\H^{1/2}\tilde\H^{-1}\bar\H^{1/2}
    \preceq \eta\sqrt\alpha(1+\epsilon_l)\I.
  \end{align*}
Now, using the fact that $\eta = \min\{\sqrt{2/\alpha},1\}$, we
conclude that:
\begin{align*}
  \|\I-\eta\bar\H^{1/2}\tilde\H^{-1}\bar\H^{1/2}\|
  &\leq
  \max\Big\{\eta\sqrt\alpha(1+\epsilon_l)-1,1-\eta\frac{1-\epsilon_l}{\sqrt\alpha}\Big\}
  \\
  &\leq\max\Big\{\sqrt2(1+\epsilon_l)-1,
    1-\frac{\sqrt2(1-\epsilon_l)}{\alpha},1-\frac{1-\epsilon_l}{\sqrt{2}}\Big\}
  \\
  &\leq \max\Big\{1-\frac1{1.8},\  1-\frac1\alpha\Big\}\leq 1-\frac1{1.8\alpha},
\end{align*}
where we used that $\epsilon_l\leq 1/c$ for a sufficiently large
constant $c>0$ such that
$\max\{\sqrt{2}(1+\epsilon_l)-1,1-\frac{1-\epsilon_l}{\sqrt 2}\}\leq 1-\frac1{1.8}$. Now, we analyze
convergence in the norm induced by $\H$, instead of $\bar\H$, by
relying again on the fact that $\bar\H\approx_{\epsilon_l}\H$,
obtaining:
\begin{align*}
  \|\Delta_1\|_{\H}
  &\leq  \frac1{\sqrt{1-\epsilon_l}}\|\Delta_1\|_{\bar\H}\leq
                     \frac1{\sqrt{1-\epsilon_l}}  \Big(1-\frac1{1.8\alpha}\Big) \|\Delta_0\|_{\bar\H}
  \\
  &\leq
\sqrt{\frac{1+\epsilon_l}{1-\epsilon_l}}\Big(1-\frac1{1.8\alpha}\Big)\|\Delta_0\|_{\H}\leq
  \Big(1-\frac1{1.9\alpha}\Big)\|\Delta_0\|_{\H},
\end{align*}
where the last step requires $\epsilon_l = 1/c\alpha$ for sufficiently
large absolute constant $c>0$.
\end{proof}

Using Lemma \ref{l:newton} to bound the first term in
\eqref{eq:decomposition}, we obtain that:
\begin{align*}
\|\Delta_t -\eta\Hbt_t^{-1}\g_t\|_{\H}
  &\leq \Big(1-\frac1{1.9\alpha}\Big)\|\Delta_t\|_{\H}.
\end{align*}

Putting everything together, we obtain the following bound for the
error of the update that uses the stochastic variance-reduced gradient
estimate:
\begin{align*}
  \|\Delta_{t+1}\|_{\H}
  &= \|\Delta_t-\eta\pbt_t\|_{\H}
  \\
  &\leq
    \|\Delta_t-\eta\Hbt^{-1}\g_t\|_{\H} +
    \eta\|\Hbt^{-1}(\g_t-\bar\g_t)\|_{\H}
  \\
  &\leq
\Big(1-\frac1{1.9\alpha}\Big)\|\Delta_t\|_{\H}
    + 2\eta\sqrt{\ch}\epsilon_g \|\Delta_0\|_{\H}\\
  &\leq \Big(1-\frac1{1.9\alpha}\Big) \|\Delta_t\|_{\H} +  3\epsilon_g \|\Delta_0\|_{\H}.
\end{align*}
Note that, as
long as $3\epsilon_g\leq \frac1{2\alpha}$ (which can be ensured by
our assumption on $m$), this implies that
$\|\Delta_{t+1}\|_{\H}\leq\|\Delta_0\|_{\H}$ and so $\x_{t+1}\in
U_f(\epsilon_l)$. Thus, our analysis can be applied recursively at
each inner iteration.
To expand the error recursion, observe that if we apply a union bound over the
high-probability events in Lemma \ref{l:gradient} for each inner
iteration $t$ using failure probability $\delta_t = \delta/t^2$, then they hold for
all $t$ with probability at least $1-\sum_{t=1}^\infty \delta/t^2\geq
1-\delta\pi^2/6$. We obtain:
\begin{align*}
  \|\Delta_{t}\|_{\H}
  &\leq \Big(1-\frac1{1.9\alpha}\Big)\|\Delta_{t-1}\|_{\H} +
    3\epsilon_g\|\Delta_0\|_{\H}
    \\
  &\leq \Big(1-\frac1{1.9\alpha}\Big)^t\|\Delta_0\|_{\H} +
  \bigg(\sum_{i=0}^{t-1}\Big(1-\frac1{1.9\alpha}\Big)^i\bigg)\cdot
3\epsilon_g\|\Delta_0\|_{\H}
  \\
  &\leq \bigg(\Big(1-\frac1{1.9\alpha}\Big)^t+9\alpha\epsilon_g\bigg)\cdot \|\Delta_0\|_{\H}.
\end{align*}
Applying Lemma \ref{l:conversion}, we convert this to convergence
in function value:
\begin{align*}
  f(\x_t)-f(\x^*)
  &\leq \frac1{1-\epsilon_l}\frac12\|\Delta_t\|_{\H}^2
  \\
  &\leq
    \frac1{1-\epsilon_l}\frac12\bigg(\Big(1-\frac1{1.9\alpha}\Big)^{t} +
    9\alpha\epsilon_g\bigg)^2\|\Delta_0\|_{\H}^2
  \\
  &\leq
    \frac{1+\epsilon_l}{1-\epsilon_l}\bigg(\Big(1-\frac1{1.9\alpha}\Big)^{2t} +
    9^2\alpha^2\epsilon_g^2\bigg)\cdot(f(\x_0)-f(\x^*))
  \\
  &\leq \bigg(\Big(1-\frac1{2\alpha}\Big)^t +
    C'\alpha^2\frac{\kappa\log(t/\delta)}{m}\bigg)
    \cdot(f(\x_0)-f(\x^*)),
\end{align*}
where $C'$ is an absolute constant, and we again used that
$\epsilon_l\leq 1/c\alpha$ for a sufficiently large $c$, thus
concluding the proof.

\subsection{Proof of Lemma \ref{l:conversion}}
\label{a:conversion}
First, we show that the Hessian at $\x\in U_f(\epsilon_l)$ is an
$\epsilon_l$-approximation of the Hessian at the optimum $\x^*$. This
is broken down into two cases, depending on which of the two
Assumptions \ref{a:lipschitz} and \ref{a:self-concordant} are
satisfied. 

Case 1: Assumption \ref{a:lipschitz} (Lipschitz Hessian). \
Using
the shorthand $\H=\nabla^2 f(\x^*)$ and the fact that strong convexity
(Assumption \ref{a:convex}) implies that $\nabla^2f(\x)\succeq \mu\I$,
we have:
\begin{align*}
  \|\H^{-1/2}(\nabla^2f(\x)-\H)\H^{-1/2}\|
    \leq
  \frac1\mu\|\nabla^2f(\x)-\H\|
  \leq \frac L\mu\|\x-\x^*\|
\leq
  \frac{L}{\mu^{3/2}}\|\x-\x^*\|_{\H}\leq \epsilon_l,
\end{align*}
which implies that $\nabla^2f(\x)\approx_{\epsilon_l}\!\H$. 

Case 2: Assumption \ref{a:self-concordant} (Self-concordance). \
The fact that $\nabla^2f(\x)\approx_{\epsilon_l}\!\H$ follows from
the following property of self-concordant functions \cite[Chapter
9.5]{boyd2004convex}, which holds when $\|\x-\x^*\|_{\H}<1$:
\[(1-\|\x-\x^*\|_{\H})^2\cdot\H\preceq\nabla^2 f(\x)\preceq (1-\|\x-\x^*\|_{\H})^{-2}\cdot\H,\]
where we again let $\H=\nabla^2 f(\x^*)$. 

We next use a version of Quadratic Taylor's Theorem, as given
below. See Theorem 3 in Chapter 2.6 of
  \cite{jerrard2018multivariable} and Chapter 2.7 in \cite{folland2002advanced}.
\begin{lemma}
  Suppose that $f:\R^d\rightarrow \R$ has continuous first and second
  derivatives. Then, for any $\a$ and $\v$, there exists $\theta\in(0,1)$
  such that:
  \begin{align*}
    f(\a+\v) = f(\a) + \nabla f(\a)^\top\v + \frac12\v^\top\nabla^2 f(\a+\theta\v)\v.
  \end{align*}
\end{lemma}
Applying Talyor's theorem with $\a=\x^*$ and $\v=\x-\x^*$, there is a
$\z=\x^*+\theta(\x-\x^*)$ such that:
\begin{align*}
  f(\x)-f(\x^*) = \frac12\|\x-\x^*\|_{\nabla^2 f(\z)}^2,
\end{align*}
where we use that $\nabla f(\x^*)=\zero$. Since we assumed that $\x\in
U_f(\epsilon_l)$, and naturally also $\x^*\in U_f(\epsilon_l)$, this
means that $\z\in U_f(\epsilon_l)$, given that $U_f(\epsilon_l)$ is
convex. Thus, using that $\nabla^2f(\z)\approx_{\epsilon_l}\H$, we have
$\|\x-\x^*\|_{\nabla^2 f(\z)}^2\approx_{\epsilon_l}\|\x-\x^*\|_{\H}^2$.

\subsection{Proof of Theorem \ref{c:least-squares}}
\label{a:ls}
In this section, we discuss how the convergence analysis of SVRN can
be adapted to using leverage score sampling when solving a least
squares task (proving Theorem \ref{c:least-squares}).

Consider an expected risk minimization problem $f(\x)=\E[\psi(\x)]$, where
$\psi=\frac1{np_I}\psi_I$ and $I$ is an index from $\{1,...,n\}$,
sampled according to some importance sampling distribution
$p$. More specifically, consider a least squares task, where the
components are given by $\psi_i(\x) =
\frac1{2}(\a_i^\top\x-y_i)^2$. Then, the overall minimization task becomes:
      \begin{align}
        \E[\psi(\x)]
        &= \E\big[\frac1{np_I}\psi_I(\x)\big] =
          \frac1{2n}\sum_{i=1}^n(\a_i^\top\x-y_i)^2.
          \label{eq:least-squares}
      \end{align}
Moreover, we have $f(\x)-f(\x^*) =
      \frac1{2n}\|\A(\x-\x^*)\|^2=\frac12\|\x-\x^*\|_{\H}^2$, where
      $\H = \nabla^2f(\x) = \frac1{n}\A^\top\A$. Also, 
      \begin{align*}
        \nabla \psi_i(\x) = (\a_i^\top\x-y_i)\a_i,\qquad\nabla^2\psi_i(\x)=\a_i\a_i^\top.
      \end{align*}
      Naturally, since the Hessian is the same everywhere for this
      task, the local convergence neighborhood $U_f$ is simply the
      entire Euclidean space $\R^d$. Let
      us first recall our definition of the condition number for this task.
      Assumption~\ref{a:convex} states that each $\psi_i$ is 
      $\lambda$-smooth, i.e.,
      $\|\nabla^2\psi_i(\x)\|=\|\a_i\|^2\leq \lambda$ and $f$ is
      $\mu$-strongly convex, i.e., 
      $\lambda_{\min}(\H)=\frac1n\sigma_{\min}^2(\A)\geq\mu$, and the condition number of the
      problem is defined as $\kappa=\lambda/\mu\geq \max_i\{n\|\a_i\|^2\}/\sigma_{\min}^2(\A)$. Can we
      reduce the condition number of this problem by importance sampling?

Consider the following naive strategy which can be applied
      directly with our convergence result. Here, we let the
      importance sampling probabilities be $p_i\propto
      \|\a_i\|^2$, so that the smoothness of the new reweighted problem
      will be $\tilde\lambda=\frac1n\sum_{i=1}^n\|\a_i\|^2$.
      In other words, it will be
      the average smoothness of the original problem, instead of the
      worst-case smoothness. Such importance sampling strategy can
      theoretically be applied to a general finite-sum
      minimization task with some potential gain, however we may
      need different sampling probabilities at each step. For least
      squares, the resulting condition number is
      $\tilde\kappa =
      \tilde\lambda/\mu = (\sum_i\|\a_i\|^2)/\sigma_{\min}^2(\A)$. This is still worse
      than what we claimed for least squares, but it is still
      potentially much better than $\kappa$.

      Next, we will show that by slightly adapting our convergence
      analysis, we can use leverage score sampling to further improve
      the convergence of SVRN for the least squares task. Recall that
      the $i$th leverage score of $\A$ is defined as
      $\ell_i=\|\a_i\|_{(\A^\top\A)^{-1}}^2=\frac1n\|\a_i\|_{\H^{-1}}^2$,
      and the laverage scores satisfy $\sum_{i=1}^n\ell_i=d$. This result will
      require showing a specialized version of Lemma \ref{l:gradient},
      which bounds the error in the variance-reduced subsampled
      gradient. In this case we show a
      bound in expectation, instead of with high probability.
      \begin{lemma}\label{l:gradient-ls}
        Suppose that $f$ defines a least squares task
        \eqref{eq:least-squares} and the sampling probabilities
        satisfy $p_i\geq \|\a_i\|_{(\A^\top\A)^{-1}}^2/(Cd)$. Then,
       $\gbh(\x)=\frac1m\sum_{i=1}^m\frac1{np_{I_i}}\nabla\psi_{I_i}(\x)$,
       where $I_1,...,I_m\sim p$, satisfies:
       \vspace{-2mm}
        \begin{align*}
          \E\,\|\gbh(\x)-\gbh(\x^*)-\nabla f(\x)\|_{\H^{-1}}^2 \leq
          C\frac {d}m\cdot\|\x-\x^*\|_{\H}^2.
        \end{align*}
      \end{lemma}
      \begin{proof}
 We define $\v_i =
        \frac1{np_{I_i}}(\nabla\psi_{I_i}(\x)-\nabla\psi_{I_i}(\x^*))-\nabla
        f(\x)$. Note that $\E[\v_i]=\zero$, so we have:
        \begin{align*}
          \E\,\|\gbh(\x)-\gbh(\x^*)-\nabla f(\x)\|_{\H^{-1}}^2
          &=\E\,\Big\|\frac1m\sum_{i=1}^m\v_i\Big\|_{\H^{-1}}^2
=\frac1m \E\,\|\v_1\|_{\H^{-1}}^2
          \\
          &\leq
            \frac1m\E\,\frac1{n^2p_{I_1}^2}\|\nabla\psi_{I_1}(\x)-\nabla\psi_{I_1}(\x^*)\|_{\H^{-1}}^2
          \\
          &=\frac1m\E\,\frac{\|\a_{I_1}\|_{\H^{-1}}^2}{n^2p_{I_1}^2}\big(\a_{I_1}^\top(\x-\x^*)\big)^2
          \\
          &\leq \frac
            {1}m\,Cd\cdot\E\,\frac{(\a_{I_1}^\top(\x-\x^*))^2}{np_{I_1}}
            = C\cdot\frac dm\|\x-\x^*\|_{\H}^2,
        \end{align*}
        where we used that $\|\a_i\|_{\H^{-1}}^2 =
        n\,\|\a_i\|_{(\A^\top\A)^{-1}}^2\leq Cndp_i$.
      \end{proof}
      Since the above bound is obtained in expectation, to insert it
      into our high probability analysis, we apply
      Markov's inequality. Namely, it holds with probability
      $1-\delta$ that:
      \begin{align*}
        \|\gbh(\x)-\gbh(\x^*)-\nabla f(\x)\|_{\H^{-1}}^2\leq
        \frac {Cd}{\delta m}\|\x-\x^*\|_{\H}^2.
      \end{align*}
      Compared to Lemma \ref{l:gradient}, the dependence on the condition number $\kappa$ is
      completely eliminated in this result.  Letting $m=n/\log(n/d)$ and the number
      of local iterations of SVRN to be $t=O(\log(n/d))$, we can
      apply the union bound argument from the proof of Theorem \ref{t:svrn} by
      letting $\delta=1/(Ct)$, so that with probability
      $1-1/C$, one stage of leverage score sampled  SVRN satisfies:
      \begin{align*}
        f(\xbt_{s+1})-f(\x^*) &\leq \rho \cdot
        \big(f(\xbt_s)-f(\x^*)\big)
        \qquad\text{for}\qquad \rho=O\Big(\frac{d\log^2(n/d)}{n}\Big).
      \end{align*}
      Alternatively, our main convergence analysis can be
      adapted (for least squares) to convergence in expectation,
      obtaining that \ $\E[f(\xbt_{s+1})-f(\x^*)] \leq \tilde\rho \cdot
        \E[f(\xbt_s)-f(\x^*)]$ \ for \ $\tilde\rho = O(d\log(n/d)/n)$.

      The time complexity stated in Theorem \ref{c:least-squares}
      comes from the fact that constructing a preconditioning matrix
      $\Hbt$ that is an $\alpha$-approximation of $\H$ with
      $\alpha=O(1)$,  together with approximating the
      leverage scores, takes $O(nd\log n + d^3\log d)$
      \cite{fast-leverage-scores},
      whereas one stage of SVRN takes $O(nd + d^2\log(n/d))$. Here, the
      preconditioning matrix can be formed by applying a $k\times n$ sketching
transformation $\S$ to the data matrix $\A$, and then computing the
Hessian estimate $\frac1n\A^\top\S^\top\S\A\approx \H$. For example, if we use the
Subsampled Randomized Hadamard Transform \cite[SRHT]{ailon2009fast},
then it suffices to use $k=O(d\log d)$. Finally, 
      the initial iterate $\xbt_0$ can be constructed using the same
      sketching transformation via the so-called sketch-and-solve
      technique \cite{sarlos-sketching}:
      \[\xbt_0 =
        \argmin_\x\|\S\A\x-\S\y\|^2.\]
      With $k=O(d\log d)$, this initial
      iterate will satisfy $f(\xbt_0)\leq O(1)\cdot f(\x^*)$, so 
      the number of iterations of SVRN needed to obtain
      $f(\xbt_s)\leq(1+\epsilon)f(\x^*)$ is only
      $s=O\big(\frac{\log(1/\epsilon)}{\log(n/d)}\big)$.
      
      We note
      that another way to implement SVRN with approximate leverage
      score sampling is to first precondition the entire least squares
      problem with a Randomized Hadamard Transform (i.e., SRHT without
      the subsampling):
      \begin{align}
        \tilde\A=\H\D\A\qquad\text{ and }\qquad\tilde\y=\H\D\y,\label{eq:rht}
      \end{align}
      where $\H$ is a
      Hadamard matrix scaled by $1/\sqrt n$ and $\D$ is a diagonal
      matrix with random sign entries. This is a popular technique in
      Randomized Numerical Linear Algebra \cite{woodruff2014sketching,DM16_CACM,dpps-in-randnla}. The cost of
      this transformation is $O(nd\log n)$, thanks to fast Fourier
      transform techniques, and the resulting least squares task is 
      equivalent to the original one, because
      $\|\tilde\A\x-\tilde\y\|^2 = \|\A\x-\y\|^2$ for all $\x$.
      Moreover, with high probability, all of the leverage scores of
      $\tilde \A$ are nearly uniform, so, after this
      preconditioning, we can simply implement SVRN with uniform
      gradient subsampling and still enjoy the condition-number-free
      convergence rate from Theorem \ref{c:least-squares}. This
      strategy is as effficient as direct leverage score
      sampling when $\A$ is a dense matrix, but it is
      less effective when we want to exploit data sparsity.

      \begin{figure*}[!t]
\vspace{-7mm}\centering     
\subfigure[Synthetic LR, $\kappa_{\A}^2 = 1$]{\label{fig:lr-cond-1}\includegraphics[width=.45\textwidth]{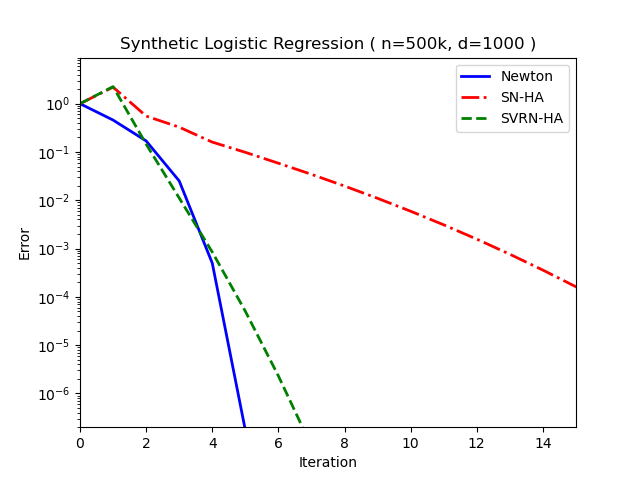}}
\subfigure[Synthetic LR, $\kappa_{\A}^2
=10$]{\label{fig:lr-cond-10}\includegraphics[width=.45\textwidth]{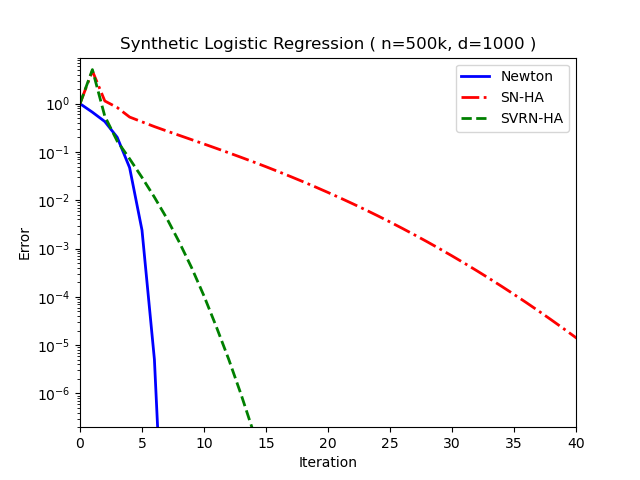}}
\\[-3mm]
\subfigure[CIFAR-10, $d = 500$]{\label{fig:lr-cond-1}\includegraphics[width=.45\textwidth]{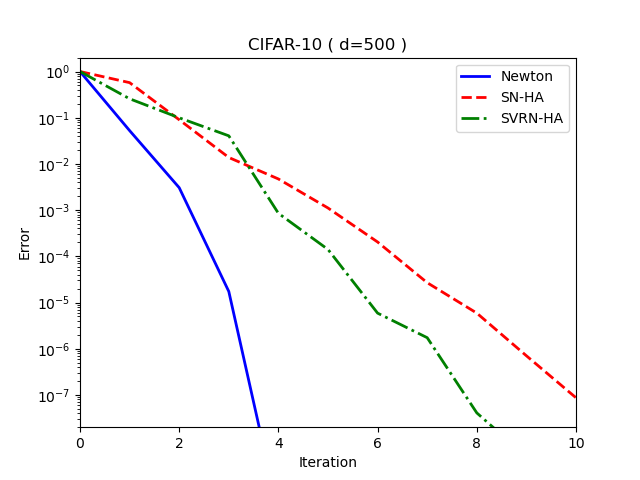}}
\subfigure[CIFAR-10, $d = 1000$]{\label{fig:lr-cond-10}\includegraphics[width=.45\textwidth]{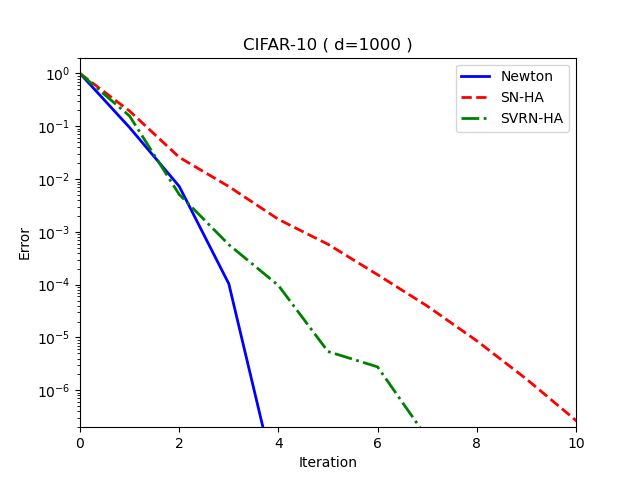}}
\caption{Convergence comparison of SVRN-HA against SN-HA and Newton
  for a synthetic logistic regression task as we vary the
  condition number of the data matrix, and for the CIFAR-10 dataset.}\label{fig:lr-cond}
\end{figure*}

\section{Further experimental details}
\label{a:experiments}

In this section we provide additional details regarding our
experimental setup in Section~\ref{s:experiments}, as well as some
further results on logistic regression with several datasets.

As a dataset, in Section \ref{s:experiments}, we used the Extended
MNIST dataset of handwritten digits \cite[EMNIST]{cohen2017emnist} with
$n=500$k datapoints.
Here, we also include results on the CIFAR-10 image dataset with
$n=50$k datapoints. Both
datasets are preprocessed in the same way:
Each image is transformed by a random features map
that approximates a Gaussian kernel having width $0.002$, and we
partitioned the classes into two labels 1 and -1. We considered two
feature dimensions: $d=500$ and $d=1000$, and we used the
regularization parameter $\gamma = 10^{-8}$. To measure the error in the convergence plots, we use
$\|\x_t-\x^*\|_{\H}^2/\|\x_0-\x^*\|_{\H}^2$, where $\H=\nabla^2
f(\x^*)$.

We next present further results, studying the convergence properties
of SVRN on synthetic datasets with varying properties, for the logistic
regression task as in \eqref{eq:lr}.
To construct our synthetic data matrices, we first generate an $n\times
d$ Gaussian matrix $\G$, and let $\G=\U\D\V$ be the reduced SVD of
that matrix (we used $n=500\text{k}$ and $d=1000$). Then, we replace diagonal matrix $\D$ with a matrix
$\tilde\D$ that has singular values spread linearly from 1 to
$\kappa_{\A}$. We then let $\A=\U\tilde\D\V$ be our data matrix. 
To generate the vector
$\y$ for logistic regression, we first draw a random vector $\x \sim\Nc(\zero,1/d\cdot\I_d)$,
and then we let $\y = \mathrm{sign}(\A\x)$.

For the least squares tasks in Section \ref{s:experiments-ls}, we generated the same synthetic matrices
$\A$, but with the
vector $\y$ generated as follows: $\y=\A\x+\xib$ where $\xib$ is the
Gaussian noise $\xib \sim \Nc(\zero,0.1\cdot\I_n)$.  Here, we
observed little difference in convergence behavior when varying
$\kappa_{\A}$ (we show the results for $\kappa_{\A}=10^3$).

\paragraph{Logistic regression with varying condition number.}

To supplement the EMNIST logistic regression experiments in
Section~\ref{s:experiments}, we present 
convergence of SVRN-HA on the CIFAR-10 dataset, as well as for the synthetic logistic
regression task while varying the squared condition number $\kappa_{\A}^2$ of the data
matrix. Note that, while $\kappa_{\A}^2$ is not the same as the
condition number of the finite-sum minimization problem, it is
correlated with it, by affecting the convexity and smoothness of
$f$. From Figure \ref{fig:lr-cond}, we observe that SVRN-HA 
outperforms SN-HA for both values of the data condition
number. However, the convergence of both algorithms gets noticeably
slower after increasing $\kappa_{\A}^2$, while it does not have as
much of an effect on the Newton's method. Given that the increased
condition number affects both methods similarly, we expect that the
degradation in performance is primarily due to worse Hessian
approximations, rather than increased variance in the gradient
estimates. This may be because we are primarily affecting the global convexity
of $f$, as opposed to the smoothness of individual components $\psi_i$. See our
high-coherence least squares experiments for a discussion of how the
smoothness of component functions affects the performances of SVRN and
SN very differently.

\section{Related work on Subsampled Newton}
\label{a:related}

In this section, we discuss several important prior works on Subsampled
Newton methods to put our results in context. Specifically, we aim to illustrate
how the Hessian approximation condition used 
in Theorems \ref{t:informal} and \ref{t:svrn}, i.e.,
$\frac1{\sqrt\alpha}\nabla^2f(\x)\preceq\Hbt\preceq\sqrt\alpha\nabla^2f(\x)$,
relates to the Hessian estimates used in this line of works
when showing fast local convergence rates. Also, in
Appendix~\ref{a:matrix-bernstein}, we show that to recover our
condition with 
$\alpha\leq 2$ via uniform Hessian subsampling, one needs
$O(\kappa\log(d))$ samples.  Throughout this section, we 
use notation from the respective references.

First, we consider the Hessian averaging method studied by
  \cite{hessian-averaging}. It is important to distinguish between the
  condition they impose on the stochastic Hessian oracle $\Hbh$, and
  the Hessian approximation guarantee that they obtain for the actual
  estimate $\Hbt_t$ that they use (and that we use in SVRN-HA). The
  stochastic oracle is only required to have a 
  sub-exponential tail (see their Assumption 2.1, and also Example 2.3
  illustrating this for Subsampled Newton). However, the actual
  estimate $\Hbt_t$ is a result of averaging many samples from that
  oracle. Their local convergence analysis is only deployed once
  enough oracle samples are averaged so that $\Hbt_t$
 achieves the approximation guarantee given in their Lemma 3.5, i.e.,
 $(1-\psi)\H_t\preceq\Hbt_t\preceq(1+\psi)\H_t$. This
 approximation guarantee is strictly stronger than ours, but it
 becomes equivalent once $\alpha\leq 2$.

 Next, we consider \cite{roosta2019sub} which analyzes a broad
  class of Subsampled Newton methods. In this paper, the most
  relevant results are Lemma~2 (Hessian approximation guarantee) and
  Theorem~5 (local convergence result). The guarantee reduces to our
  condition with $\alpha\leq 2$, with the only difference being that their Hessian
  approximation is restricted to the “cone of feasible directions”,
  defined in (3). This restriction is only present in a constrained
  optimization setting (we focus on unconstrained optimization). The
  lemma also shows that the required Hessian sample size is again
  larger than the condition number of the problem (their condition
  number $\kappa_1$ matches our $\kappa$ for unconstrained
  optimization). 

  Next, we look at \cite{bollapragada2018exact} which presents a
  convergence analysis of Subsampled Newton under slightly different
  assumptions. Here, the key statements for local convergence
  are Lemma 2.4 in the journal version (Lemma 2.3 in arxiv), and
  equation (2.17). The lemma gives an approximation guarantee for the
  subsampled Hessian. This guarantee is in some sense weaker than our
  condition, because it only requires the Hessian approximation to be
  good in one direction, i.e., $w_k-w^*$ in their notation. However,
  examining the convergence bound in (2.17), for the local convergence
  analysis to hold, the Hessian sample size must still satisfy
  $|S_k|\geq\sigma^2/\bar\mu^2$ where $\sigma^2$ is effectively the
  upper bound on the component Hessians (potentially as large as our  
$\lambda$-smoothness) and $\bar\mu$ is the lower bound on the
component Hessians. The latter is effectively a component-wise strong
convexity constant, which can be much smaller (i.e., worse) than our
global strong convexity  $\mu$. In summary, their Hessian
approximation condition for local convergence analysis, while slightly
different, also requires the Hessian sample size to be larger than a
condition number of the problem. Their condition number can be much
larger than our condition number $\kappa$, or even infinite (for
problems as simple as least squares), and is less standard in the
literature. 

Finally, we examine \cite{erdogdu2015convergence}, where the
  authors consider Subsampled Newton with a possibly low-rank
  approximation of the Hessian. The most relevant result  in that work
  is Lemma 3.1. Here, the standard version of Subsampled Newton is
  recovered when we let  $Q^t=H_{S_t}^{-1}$. Then, the standard
  Hessian approximation condition appears implicitly through the fact
  that  $\xi_1^t$ has to be less than 1 for the bound to be
  non-vacuous. To see the condition more clearly, we point to Equation
  (B.1) in the appendix (Equation A.1 in the arxiv version), which
  requires that  $\|Q^t\|\cdot\|H_{S_t}-H\|<1$. For
  $Q^t=H_{S_t}^{-1}$, this is essentially equivalent to our
  condition with $\alpha\leq 2$. Also, from the bound in Lemma 3.1, we see that once again
  the condition requires Hessian sample size to be larger than a
  condition number of the problem (which is for them $K\|Q^t\|$, and
  after some effort, this can be seen as comparable to our condition
  number). In the main algorithm of the paper, NewSamp, the authors
  aim to reduce the required sample size by using a different $Q^t$
computed from a low-rank approximation of
 $H_{S_t}$. This roughly corresponds to constructing a Hessian
 $\alpha$-approximation with $1\ll\alpha\ll\kappa$.

\section{Omitted proofs}
\label{a:proofs}

Here, we include the proofs of the auxiliary results stated in the paper.
First, we discuss in
detail the global convergence analysis of SVRN-HA
(Theorem~\ref{t:global}). Then, we illustrate how the Hessian approximation
required in Theorem~\ref{t:svrn} can be obtained via
subsampling.

\subsection{Global convergence of SVRN-HA}

Here, we show how the proof of Theorem \ref{t:global}, i.e., global
convergence of SVRN-HA (Algorithm \ref{alg:svrn}), follows from
global convergence analysis of Hessian averaging
\cite{hessian-averaging}. They show in Lemma 3.5 that if we were to
run the global phase of SVRN-HA exclusively, then for any
$\epsilon,\delta\in(0,1)$ there is $T:=T(\epsilon,\delta)$ such that with probability
$1-\delta$ for all $s\geq T$ we have
$\xbt_s\in U_f(\epsilon)$, $\Hbt_s\approx_\epsilon\nabla^2 f(\xbt_s)$,
and $\eta_s=1$. This means that, for any $\epsilon$, the probability
that the above event does not happen with any $T<\infty$ is less than any
$\delta>0$, so it must be $0$. This implies that SVRN-HA will
eventually switch to the local phase (i.e., to SVRN). Note that it is
possible that the switch will occur before the local neighborhood and
Hessian approximation conditions are met. But if this causes SVRN to
produce a poor descent direction, it will be caught by the line
search (resulting in $\eta_s<1$) and the method will simply revert back
to the global phase. Eventually, the global phase will ensure that
both conditions are met, and we can rely on Theorem~\ref{t:svrn} for
the local convergence analysis.

\subsection{Hessian approximation via subsampling}
\label{a:matrix-bernstein}

Here, we illustrate how the Hessian $\alpha$-approximation
condition \eqref{eq:alpha}, used in  Theorems~\ref{t:informal} and \ref{t:svrn}, can be obtained via
uniformly subsampling $O(\kappa\log d)$ component Hessians. This
result follows from Bernstein's concentration inequality for
random matrices, given below \cite{matrix-tail-bounds}. 
\begin{lemma}[Matrix Bernstein's inequality]\label{l:matrix-bernstein}
  Let $\Z_1,...,\Z_k$ be independent random symmetric $d\times d$ matrices 
    such that $\frac1k\sum_i\E[\Z_i]=\bar\Z$. Suppose that:
    \begin{align*}
    \big\|\frac1k\sum_i\E[(\Z_i-\E[\Z_i])^2]\big\|\leq\bar\sigma^2\quad\text{and}\quad  \|\Z_i-\E[\Z_i]\|\leq R.
    \end{align*}
    Then, for any $\epsilon\geq 0$
    \begin{align*}
      \Pr\bigg(\Big\|\frac1k\sum_{i=1}^k\Z_i-\bar\Z\Big\|\geq \epsilon\bigg)\leq
      2d\cdot\exp\Big(-\frac{\epsilon^2k/2}{\bar\sigma^2+\epsilon R/3}\Big).
    \end{align*}
  \end{lemma}

  We are now ready to show the approximation guarantee for a
  subsampled Hessian estimate.
  \begin{lemma}
    Suppose Assumption \ref{a:convex}, and let $\Dc$ be the
    sampling distribution for component functions $\psi$, as in
    Theorem~\ref{t:svrn}.
    Let $\psi_1,...,\psi_k\sim\Dc$ be i.i.d. samples from this
    distribution. There is an absolute constant $c$ 
    such that for 
    any $\x\in\R^d$, with
    probability $1-\delta$, the matrix
    $$\Hbt=\Big(1+\frac\gamma\mu\Big)^{-1/2}\Big(\frac1k\sum_{i=1}^k\nabla^2\psi_i(\x)
    +\gamma\I\Big),\qquad\text{with}\quad \gamma =
    \max\Big\{12\lambda\log(2d/\delta)/k,\mu\Big\},$$
    is an $\alpha$-approximation of $\nabla^2f(\x)$ as in \eqref{eq:alpha}
    with
    $$\alpha = 1 + O\Big(\kappa\log(d/\delta)/k + \sqrt{\kappa\log(d/\delta)/k}\Big).$$
  \end{lemma}

    \begin{proof}
    Let $\H_\gamma=\nabla^2f(\x)+\gamma\I$ for some $\gamma\geq 0$. We will use Lemma \ref{l:matrix-bernstein} with $\Z_i =
    \H_\gamma^{-1/2}\nabla^2\psi_i(\x) \H_\gamma^{-1/2}$. First, note that
    $\E[\nabla^2\psi_i(\x)]=\nabla^2f(\x)$ so that
    $\bar\Z=\E[\Z_i]=\H_\gamma^{-1/2}\nabla^2f(\x)\H_\gamma^{-1/2}\preceq
    \I$. Next, using that
    $\|\H_\gamma^{-1}\|\leq 1/(\mu+\gamma)$, we compute the boundedness parameter $R$: 
    \begin{align*}
      \|\Z_i-\E[\Z_i]\|\leq 
      \|\H_\gamma^{-1/2}\nabla^2\psi_i(\x)\H_\gamma^{-1/2}\|+1\leq \frac{2\lambda}{\mu+\gamma}=:R.
    \end{align*}
    Now, we similarly bound the variance parameter $\bar\sigma^2$:
    \begin{align*}
      \big\|\E[(\Z_i-\E[\Z_i])^2]\big\|\leq \big\|\E[\Z_i^2]\big\|\leq
      \big\|\E\big[\|\Z_i\|\Z_i\big]\big\|\leq \frac{\lambda}{\mu+\gamma}\|\E[\Z_i]\|\leq \frac{\lambda}{\mu+\gamma}=:\bar\sigma^2.
    \end{align*}
    Thus, Lemma \ref{l:matrix-bernstein} implies that if $k\geq
    \frac{3\lambda}{\mu+\gamma}\log(2d/\delta)/\epsilon^2$ then with probability
    $1-\delta$ the Hessian estimate
    $\tilde\H=\big(1+\frac\gamma\mu\,\big)^{-1/2}\big(\frac1k\sum_{i=1}^k\nabla^2\psi_i(\x)+\gamma\I\big)$ satisfies:
    \begin{align*}
      \Big\|\H_\gamma^{-1/2}\Big(1+\frac\gamma\mu\,\Big)^{1/2}\,\Hbt\H_\gamma^{-1/2} - \I\Big\| =
      \Big\|\frac1k\sum_{i=1}^k\Z_i - \bar\Z\Big\|\leq \epsilon.
    \end{align*}
    We can rewrite this as:
    $$(1-\epsilon)\I\preceq
    \H_\gamma^{-1/2}\Big(1+\frac\gamma\mu\,\Big)^{1/2}\,\Hbt\H_\gamma^{-1/2}\preceq(1+\epsilon)\I,$$
    which is equivalent to:
    $$(1-\epsilon) \Big(1+\frac\gamma\mu\,\Big)^{-1/2}\,\H_\gamma\preceq
    \Hbt\preceq(1+\epsilon) \Big(1+\frac\gamma\mu\,\Big)^{-1/2}\,\H_\gamma.$$
    Moreover, note that the regularized Hessian $\H_\gamma$ satisfies:
    \begin{align*}
      \nabla^2 f(\x) \preceq \H_\gamma = \nabla^2 f(\x)+\frac\gamma\mu\,\mu\I \preceq \Big(1+\frac\gamma\mu\Big)\nabla^2 f(\x).
    \end{align*}
Putting this together, and assuming that $\epsilon\leq 1/2$, for $k\geq
\frac{3\lambda}{\mu+\gamma}\log(2d/\delta)/\epsilon^2$ we get: 
    \begin{align*}
\frac1{\sqrt\alpha}\nabla^2f(\x)\preceq \tilde\H\preceq
      \sqrt\alpha\nabla^2 f(\x),\qquad\text{with}\quad\alpha =
      (1+2\epsilon)^2\Big(1+\frac\gamma\mu\Big). 
    \end{align*}
    Thus, if we set
    $\gamma=\max\{12\lambda\log(2d/\delta)/k,\mu\}$, then there are
    two cases:
    \begin{enumerate}
      \item If $k\leq 12\kappa\log(2d/\delta)$, then we have
    $k=12\frac\lambda\gamma\log(2d/\delta)\geq
    \frac{3\lambda}{\mu+\gamma}\log(2d/\delta)/\epsilon^2$ for $\epsilon=1/2$, which
    implies that $\tilde\H$ is an $\alpha$-approximation
    of $\nabla^2f(\x)$ with $\alpha = O(\kappa\log(d/\delta)/k)$.
  \item If $k\geq 12\kappa\log(2d/\delta)$, then let $\epsilon =
    \sqrt{\gamma/2\mu}$, so that we have
    $k=12\frac\lambda\gamma\log(2d/\delta) = 3\kappa\log(2d/\delta)/\epsilon^2$,
    which
    implies that $\tilde\H$ is an $\alpha$-approximation of
    $\nabla^2f(\x)$ with $\alpha = (1+2\epsilon)^3 = O(1+\sqrt{\kappa\log(d/\delta)/k})$.
  \end{enumerate}
  \end{proof}

\end{document}